\documentclass[11pt,a4paper]{amsart}

\title[The $\debar$-equation for $(p,q)$-forms on a non-reduced analytic space]{The 
$\debar$-equation for $(p,q)$-forms on a non-reduced analytic space}

\author[M. Andersson, R. L\"{a}rk\"{a}ng, M. Lennartsson, H. Samuelsson Kalm]{Mats Andersson \& Richard L\"{a}rk\"{a}ng \& Mattias Lennartsson \\
\& H{\aa}kan Samuelsson Kalm}
\thanks{The first and second author was partially supported by the Swedish Research Council.}
\address{Department of Mathematical Sciences, Division of Algebra and Geometry, University of Gothenburg and 
Chalmers University of Technology, SE-412 96 G\"{o}teborg, Sweden}
\email{matsa@chalmers.se, larkang@chalmers.se, matlen@chalmers.se, hasam@chalmers.se}

\date{\today}

\usepackage[english]{babel}
\usepackage{amsmath,calligra}
\usepackage{amsthm,amssymb,latexsym}
\usepackage[all,cmtip]{xy}
\usepackage{url,ae}
\usepackage{t1enc}
\usepackage{mathrsfs}
\usepackage{graphicx}
\usepackage{geometry}
\newtheorem{proposition}{Proposition}[section]
\newtheorem{theorem}[proposition]{Theorem}
\newtheorem{lemma}[proposition]{Lemma}
\newtheorem{corollary}[proposition]{Corollary}

\theoremstyle{definition}
\newtheorem{definition}[proposition]{Definition}
\newtheorem{example}[proposition]{Example}
\newtheorem{remark}[proposition]{Remark}

\numberwithin{equation}{section}

\DeclareMathOperator{\Hom}{\mathscr{H}\text{\kern -3pt {\calligra\Large om}}\,}
\DeclareMathOperator{\Ext}{\mathscr{E}\text{\kern -3pt {\calligra\Large xt}}\,\,}
\DeclareMathOperator{\Image}{\mathscr{I}\text{\kern -3pt {\calligra\Large m}}\,}
\DeclareMathOperator{\Kernel}{\mathscr{K}\text{\kern -3pt {\calligra\Large er}}\,}
\newcommand{\C}{\mathbb{C}}
\newcommand{\Cu}{\mathscr{C}}
\newcommand{\debar}{\bar{\partial}}
\newcommand{\A}{\mathscr{A}}
\newcommand{\Bsheaf}{\mathscr{B}}
\newcommand{\F}{\mathscr{F}}

\newcommand{\HH}{\mathscr{H}}

\newcommand{\J}{\mathcal{J}}
\newcommand{\E}{\mathscr{E}}

\newcommand{\W}{\mathcal{W}}
\newcommand{\PM}{\mathcal{PM}}

\newcommand{\hol}{\mathscr{O}}

\newcommand{\V}{\mathcal{V}}

\newcommand{\CH}{\mathscr{C} \kern -2pt \mathscr{H}}

\newcommand{\Om}{\mathit{\Omega}}

\newcommand{\Ba}{{\scalebox{1.4}{$\omega$}}}

\newcommand{\Xpreg}{X_{p\text{-reg}}}
\newcommand{\VV}{\mathcal{V}}

\def\newop#1{\expandafter\def\csname #1\endcsname{\mathop{\rm #1}\nolimits}}
\newop{span}

\begin{document}
\nocite{*}
\bibliographystyle{plain}

\begin{abstract}
On any pure $n$-dimensional, possibly non-reduced,  analytic space $X$ we introduce the sheaves 
 $\E_X^{p,q}$ of smooth $(p,q)$-forms and certain extensions $\A_X^{p,q}$ of them
such that the corresponding Dolbeault complex is exact, i.e., the $\debar$-equation is locally 
solvable in $\A_X$. The sheaves  $\A_X^{p,q}$ are modules over the  
smooth forms, in particular, they are fine sheaves. 
We also introduce certain sheaves $\Bsheaf_X^{n-p,n-q}$ of currents on $X$ that are 
dual to $\A_X^{p,q}$ in the sense of Serre duality. More precisely, we show that the compactly supported Dolbeault
cohomology of $\Bsheaf^{n-p,n-q}(X)$ in a natural way is the dual of the Dolbeault cohomology of $\A^{p,q}(X)$.
\end{abstract}

\maketitle
\thispagestyle{empty}

\section{Introduction}

It is natural to try to find concrete realizations of abstract objects like
sheaf cohomology groups and their duals.
On a smooth complex manifold $X$ of dimension $n$ the Dolbeault--Grothendieck lemma states that
the Dolbeault complex, 
\begin{equation}\label{sol1}
0\to \Om^p_X \to \E_X^{p,0}\stackrel{\debar}{\to}\E_X^{p,1}\to \cdots,
\end{equation}
is a fine resolution of $\Om_X^p$, and by standard arguments it follows that we have the representation
\begin{equation}\label{sol2}
H^{p,q}(X):=H^q(X,\Om^p_X)\simeq H^q(\E^{p,\bullet}(X),\debar).
\end{equation}
If $X$ is compact, then the duals of these groups are represented by
$H^{n-p,n-q}(X)$ via the non-degenerate pairing
\begin{equation}\label{sol4}
H_{}^{p,q}(X)\times H_{}^{n-p,n-q}(X)\to \C,  \quad ([\phi],[\psi])\to \int_X \phi\wedge \psi,
\end{equation}
where $\phi$ and $\psi$ are $\debar$-closed $(p,q)$ and $(n-p,n-q)$-forms, respectively.   
There are analogues of this so-called Serre duality even when $X$ is not compact.  

\smallskip
If $X$ is a non-smooth reduced analytic space, then the complex  \eqref{sol1} has a meaning  but it is
not exact in general except at $q=0$.  Thus the direct analogue of  \eqref{sol2} does not hold.
However,  
there are fine sheaves $\A_X^{p,q}$ of $(p,q)$-currents,
introduced in  \cite{AS} for $p=0$ and in \cite{SK} for $p\ge 0$,
that coincide with $\E_X^{p,q}$ on $X_{reg}$, such that 
\begin{equation}\label{sol5}
0\to \Om^p_X \to \A_X^{p,0}\stackrel{\debar}{\to}\A_X^{p,1}\to \cdots
\end{equation}
are fine resolutions of\footnote{In this paper $\Om^p_X$ denotes the sheaf of K\"{a}hler differential $p$-forms
modulo torsion.} $\Om^p_X$.  This leads to the representation
\begin{equation}\label{sol6}
H^{p,q}(X):=H^q(X,\Om^p_X)\simeq H^q(\A^{p,\bullet}(X),\debar).
\end{equation}
In the non-smooth case however the duality is more involved.  Let $\Ba^p_X$ be the sheaves of 
meromorphic $(p,0)$-forms which are $\debar$-closed considered as currents on $X$.  
They were first introduced by Barlet in \cite{Barlet} in a slightly different way;
see also \cite{HP}.
In \cite{RSW, SK} 
were introduced fine sheaves $\Bsheaf_X^{p,q}$ of $(p,q)$-currents, that are smooth on $X_{reg}$,
with the following properties:
For each $p$ we have a complex
\begin{equation}\label{sol3}
0\to \Ba^p_X \to \Bsheaf_X^{p,0}\stackrel{\debar}{\to}\Bsheaf_X^{p,1}\to \cdots
\end{equation}
such that, given that $X$ is compact, 
$H^{n-q}(\Bsheaf^{n-p,\bullet}(X),\debar)$ is the dual of $H^{p,q}(X)$, realized via the non-degenerate pairing 
\begin{equation}\label{sol41}
H^{p,q}(X)\times H^{n-q}(\Bsheaf^{n-p,\bullet}(X),\debar)\to \C,  \quad ([\phi],[\psi])\to \int_X \phi\wedge \psi,
\end{equation}
where $\phi$ and $\psi$ are $\debar$-closed currents in $\A^{p,q}(X)$ and
$\Bsheaf^{n-p,n-q}(X)$, respectively.  
The complex \eqref{sol3} is exact at $q=0$ but it is a resolution of $\Ba_X^p$ if and only if $\Om^p_X$
is Cohen--Macaulay. 

\smallskip
The aim of this paper is to extend these results to the case when $X$ is a non-reduced analytic space
of pure dimension $n$.  Already in \cite{AL} were defined a resolution of the structure sheaf
$\hol_X$, that is, \eqref{sol5} for $p=0$,  and as a consequence a representation \eqref{sol6}
for $p=0$.  We thus have to extend this representation to $p\ge 0$ and find analogues of \eqref{sol3} and \eqref{sol41}.

\smallskip
Let us describe various forms and currents on our non-reduced $X$.  First recall that locally we have an embedding 
$i\colon X\to  D\subset\C^N$ and a surjective sheaf mapping $i^*\colon \hol_D^p\to \hol_X^p$.
This means more concretely that we have an ideal sheaf  $\J_X\subset\hol_D$ with zero set $X_{red}$
such that $i^*$ is the natural mapping  $\hol_D\to \hol_D/\J_X\simeq \hol_X$.
There are similar surjective mappings
$i^*\colon \Om_D^p\to \Om_X^p$ for $p\ge 1$. Moreover, we have the $\hol_X$-sheaves
$\E_X^{p,*}$ of smooth $(p,*)$-forms and natural surjective mappings
$i^*\colon \E_D^{p,*}\to \E_X^{p,*}$.  It turns out that $i^*$ is a ring homomorphism as usual  
so that we have natural products  
\begin{equation}\label{par}
\E_X^{p,q} \times \E_X^{p',q'}\to \E_X^{p+p',q+q'}, \quad (\phi,\psi)\mapsto \phi\wedge\psi.
\end{equation}

We define the sheaf $\Cu_X^{p,q}$ of $(p,q)$-currents on $X$ as the dual of the space of compactly supported sections of
$\E_X^{n-p,n-q}$. Given 
the embedding $i\colon X\to  D\subset\C^N$ we have natural injective mappings
$i_*\colon \Cu_X^{p,q}\to \Cu_D^{N-n+p,N-n+q}$ so that the elements in $\Cu_X^{p,q}$ are 
identified with the ordinary $(N-n+p,N-n+q)$-currents in $D$ that vanish on $\Kernel i^*$.  
In view of \eqref{par} we have natural products   
\begin{equation}\label{par2}
\E_X^{p,q} \times \Cu_X^{p',q'}\to \Cu_X^{p+p',q+q'}, \quad (\phi, u)\mapsto \phi\wedge u.
\end{equation}


We are mainly interested in subsheaves $\W_X^{p,q}$ of $\Cu_X^{p,q}$ where the elements
have a certain regularity property; \eqref{par2} holds also with $\Cu_X$ replaced by $\W_X$.
The subsheaf of $\debar$-closed members of $\W_X^{p,0}$ 
are denoted by $\Ba_X^p$;  they are natural extensions to our non-reduced
space $X$ of the  Barlet sheaves. 

We are also interested in another class of non-smooth forms $\V_X^{p,q}$, which however are fundamentally
different from $\Cu_X^{p,q}$. The sheaves $\V_X^{p,q}$ are extensions of $\E_X^{p,q}$ and contain
for instance principal values of meromorphic forms. Generically on $X$ elements in $\V_X^{p,q}$
are weak limits of elements in $\E_X^{p,q}$.

\begin{remark}\label{skrutt}
Notice that when $X$ is reduced we have the inclusion
$\Om_X^p\subset \Ba_X^p$ with equality if $X$ is smooth. In the non-reduced case
there is no such relation at all since the elements in $\Ba_X$, although holomorphic, are
dual objects whereas elements in $\Om_X^p$ have no natural interpretation as dual objects.
However, 
if $\phi$ is in $\Om_X^p$ and $\mu$ in $\Ba_X^{p'}$, then $\phi\wedge \mu$ is in $\Ba_X^{p+p'}$.
\end{remark}

Here is our first main theorem.


\smallskip

\noindent {\bf Theorem A.}
\emph{Let $X$ be a non-reduced analytic space of pure dimension $n$. For each $p\ge 0$ there
are fine\footnote{As in the reduced case, a sheaf is ``fine'' if it is closed under multiplication
by smooth functions.}  subsheaves $\A_X^{p,q}$ of $\V_X^{p,q}$ that coincide with $\E_X^{p,q}$
generically on $X$, such that \eqref{sol5} is a resolution of $\Om_X^p$.}

\smallskip

\noindent As an immediate corollary we get the representation \eqref{sol6} of sheaf cohomology.

\smallskip

For our second main theorem we must introduce an intrinsic notion of integration over $X$.
For each $(n,n)$-current $u$ on $X$ there is a well-defined integral
$$
\int_X u.
$$
Given a local embedding as before and assuming that $u$ has support in $D\cap X$
it is defined as the integral of $i_*u$ over $D$.






\smallskip

\noindent {\bf Theorem B.}
\emph{Let $X$ be a non-reduced analytic space of pure dimension $n$. Moreover assume that $X$ is compact.
There are fine subsheaves
$\Bsheaf^{p,q}$ of $\W_X^{p,q}$ such that} 

\smallskip
\noindent
\emph{(i) \eqref{sol3} is a complex,}

\smallskip
\noindent
\emph{(ii) \eqref{sol3}  is exact 
if and only if  $\Om_X^p$ is Cohen--Macaulay,}

\smallskip
\noindent
\emph{(iii)  the products $\phi\wedge \mu$ for $\phi$ in $\A_X^{p,*}$ and $\mu$ in $\Bsheaf_X^{n-p,*}$ are well-defined in $\W_X^{n,*}$,}

\smallskip
\noindent
\emph{(iv)  the pairing \eqref{sol41}  is well-defined and non-degenerate so that 
$H^{n-q}(\Bsheaf^{n-p,\bullet}(X),\debar)$ is the dual of $H^{p,q}(X)$.}

\smallskip 

There are variants of Theorem B even when $X$ is not compact, see Section~\ref{serresektion}.

\bigskip
The construction of the new sheaves on $X$ relies on the ideas in 
the previous papers \cite{AS, SK, AL}. The proofs of Theorems A and B rely on
explicit Koppelman formulas for the $\debar$-equation. The main novelty in this paper 
is the adaption of the ideas in \cite{AL} to the framework in \cite{SK}. We also believe that 
the non-reduced point of view sheds new light on Serre duality, even in the reduced case, cf.\ Remark~\ref{skrutt}.
Finally, we think that the notions and results in this paper may serve as tools for doing analysis on non-reduced spaces.

\smallskip

The paper is organized as follows. The main objects are introduced in Sections~\ref{holo-smooth-curr} and \ref{Vsektion} 
and their basic properties are proved. In the rather technical Section~\ref{intop} the integral operators used in the 
Koppelman formulas are defined and their basic mapping properties are shown.
The sheaves $\A_X^{p,*}$ and $\Bsheaf_X^{n-p,*}$ are introduced in Section~\ref{AoBsektion} and Theorem A as well as 
Koppelman formulas are proved. In Section~\ref{serresektion} we show Theorem B and in Section~\ref{Exsektion} some further
examples are given.

\section{Preliminaries}\label{prelim}
Throughout this paper, unless otherwise said, $\J$ is a coherent pure $n$-dimensional ideal sheaf in a domain $D\subset\C^N$, 
$Z$ is the zero set of $\J$,
$i\colon X\hookrightarrow D$ is the (possibly) non-reduced analytic subspace with structure sheaf $\hol_{D}/\J$,
and $\kappa=N-n$. 

Let $\iota \colon Z\to D$ be the inclusion. The sheaf of smooth $(p,q)$-forms on $Z$ is $\E_Z^{p,q}:=\E_D^{p,q}/\Kernel \iota^*$.
It is well-known that this is an intrinsic notion, i.e., it does not depend on the embedding $Z\to D$. The space of $(n-p,n-q)$-currents 
on $Z$ is defined as the dual
of $\E_Z^{p,q}$. More concretely, $(p,q)$-currents on $Z$ can  
be identified via $\iota_*$ with $(\kappa+p,\kappa+q)$-currents $\mu$ in $D$ such that $\J_Z\mu=d\J_Z\mu=\bar\J_Z\mu=d\bar\J_Z\mu=0$.
If $\pi\colon Z'\to Z$ is proper, $\mu$ a current on $Z'$, and $\psi$ is smooth on $Z$, then
\begin{equation}\label{projformel}
\pi_*(\pi^*\psi\wedge\mu)=\psi\wedge \pi_*\mu.
\end{equation}

In \cite{AW2}, see also \cite{AS}, was introduced the sheaf $\PM_Z$ of \emph{pseudomeromorphic currents}. A current $\tau$ in $U\subset\C^N$
is an elementary pseudomeromorphic current if $\tau=\varphi\wedge\tau'$, where $\varphi$ is smooth with compact support in $U$
and $\tau'$ is the tensor product of one-variable currents $1/z_k^{m_k}$ and $\debar(1/z_\ell^{m_\ell})$. 
If $Z$ is smooth, then, \cite[Theorem~2.15]{AWdirect}, 
a current on $Z$ is pseudomeromorphic if and only if it is a locally finite sum of currents of the form
$f_*\tau$, where $f\colon U\to Z$ is holomorphic, $U\subset \C^N$, and 
$\tau$ is elementary. If $Z$ has singularities the definition is slightly more involved. Pseudomeromorphic currents are
closed under $\debar$ and direct images of modifications, simple projections, and open inclusions.

\begin{example}
Recall that a current on $Z$ is semi-meromorphic if it is of the form $\varphi/s$, where $s$ is a generically non-vanishing section of
some line bundle $L$ and $\varphi$ is a smooth form with values in $L$. If $|\cdot |$ is any Hermitian metric on $L$, then
$\chi(|s|^2/\epsilon)\varphi/s\to\varphi/s$ as currents,
where $\chi$ is a smooth approximation of the characteristic function of $[1,\infty)\subset\mathbb{R}$.
Semi-meromorphic currents, and $\debar$ of such, are sections of $\PM$.  
\end{example}

We refer to \cite{AWdirect} for properties of pseudomeromorphic currents. If $V=\{h=0\}$ for some holomorphic tuple $h$ 
in $D$ and $\mu\in\PM(D)$, then
\begin{equation}\label{tabasco}
\mathbf{1}_{D\setminus V}\mu:=\lim_{\epsilon\to 0}\chi(|h|^2/\epsilon) \mu.
\end{equation}
The limit \eqref{tabasco} exists, is in $\PM_D$, and is independent of such $h$ and $\chi$. 
Set 
\begin{equation*}
\mathbf{1}_{V}\mu:=\mu-\mathbf{1}_{D\setminus V}\mu.
\end{equation*}
If $\pi\colon \widetilde D\to D$ is a modification or a simple projection and $\tau\in\PM(\widetilde D)$
has compact support in the fiber direction, then
\begin{equation}\label{SEPprojformel}
\mathbf{1}_V\pi_*\tau = \pi_*(\mathbf{1}_{\pi^{-1}V}\tau).
\end{equation}
If $\mu\in\PM_D$ has support in $Z$, then 
\begin{equation}\label{slurp}
\bar\J_Z\mu=d\bar\J_Z\wedge\mu=0.
\end{equation}

\noindent {\bf Dimension principle.} \emph{If $\mu\in \PM_Z$ has bidegree $(*,q)$ and support in a subvariety $V\subset Z$ 
such that $\text{codim}_Z\, V>q$, then $\mu=0$.}

\medskip

A current $\mu\in\PM_D$ with support in $Z$ has the \emph{standard extension property}
(SEP) with respect to $Z$ if $\mathbf{1}_V\mu=0$ for all germs of analytic sets $V$ in $D$ intersecting $Z$ properly.
The subsheaf of $\PM_D$ of $(N,*)$-currents with support in $Z$ and the SEP with respect to $Z$ is denoted $\W^{Z,*}_D$.
The subsheaf of $\PM_Z$ of pseudomeromorphic currents on $Z$ with the SEP with respect to $Z$ is denoted by $\W_Z$.

\begin{remark}\label{gnabbigare}
We will frequently consider $\Hom$-sheaves. For instance a sheaf like $\Hom_{\hol_D}(\Om^p_D,\W_D^{Z,*})$
can in a natural way be identified with sheaves of currents of bidegree
$(N-p,*)$ by
\begin{equation*}
\W_D^{Z,(N-p,*)} \xrightarrow{\sim} \Hom_{\hol_D}(\Om^p_D,\W_D^{Z,*}), \quad
\mu\mapsto (\varphi\mapsto \varphi\wedge\mu),
\end{equation*}
where we temporarily let $\W_D^{Z,(N-p,*)}$ denote the sheaf of pseudomeromorphic $(N-p,*)$-currents
in $D$ with support on $Z$ and the SEP with respect to $Z$. It is clear that if $\varphi\wedge\mu=0$ for all
$\varphi\in\Om_D^p$, then $\mu=0$. Hence, the mapping is injective. To see that it is surjective, let $\{dz_I\}$ be a basis
of $\Om_D^p$ and let $\{\partial/\partial z_I\}$ be the dual basis. If $u\in\Hom_{\hol_D}(\Om^p_D,\W_D^{Z,*})$ then
$u(dz_I)\in\W_D^{Z,*}$ and so, by \cite[Theorem~3.7]{AWdirect}, there are $u_I\in\W_D^{Z,(0,*)}$ such that
$u(dz_I)=dz\wedge u_I$, where $dz=dz_1\wedge\cdots\wedge dz_N$.
Define $\mu_I\in\W_D^{Z,(N-p,*)}$ by $\mu_I=\pm(\partial/\partial z_I\lrcorner dz)\wedge u_I$,
where $\pm$ is chosen so that $dz_I\wedge\mu_I=dz\wedge u_I$. Setting $\mu=\sum_I\mu_I$ it is straightforward to check
that $\varphi\wedge\mu=u(\varphi)$ for all $\varphi\in \Om_D^p$
since  $\{dz_I\}$ is a basis of $\Om_D^p$.

In this paper we will use the $\Hom$-notation but, keeping the identification in mind, we will for a $\Hom$-element $\mu$
write $\varphi\wedge\mu$ (or possibly $\mu\wedge\varphi$) instead of $\mu(\varphi)$.
\end{remark}
 
Suppose that $Z$ is smooth and that we have local coordinates $(z,w)$ centered at some $x\in Z$ such that 
$Z=\{w=0\}$.  
If $\mu\in\W_D^{Z,*}$, then there is a unique representation
\begin{equation}\label{grymta}
\mu=\sum_\alpha \mu_\alpha(z)\wedge \debar\frac{dw}{w^{\alpha+\mathbf{1}}},
\end{equation}
where the products are tensor products, $\mu_\alpha(z)\in\W_Z^{n,*}$, and $\debar(dw/w^{\alpha+\mathbf{1}})$ is shorthand for 
$\debar(dw_1/w_1^{\alpha_1+1})\wedge \debar(dw_2/w_2^{\alpha_2+1})\wedge\cdots$; see \cite[Proposition~2.5]{AL}.
By that same proposition,
\begin{equation}\label{klippning}
\mu_\alpha(z)=\pi_*(w^\alpha\mu)/(2\pi i)^\kappa,
\end{equation}
where $\pi(z,w)=z$ and $w^\alpha=w_1^{\alpha_1}w_2^{\alpha_2}\cdots$.
By \cite[Proposition~3.12, Theorem~3.14]{AWreg} we have

\begin{proposition}\label{AWprop}
If $u,\mu_1,\ldots,\mu_\ell\in\W_Z^{n,*}$ and $u=0$ on the set where all $\mu_j$ are smooth,
then $u=0$.
\end{proposition}
 
The sheaf $\CH_D^Z$ of \emph{Coleff-Herrera currents} with support on $Z$ was introduced by Bj\"{o}rk, see \cite{jebAbel}.
An $(N,\kappa)$-current $\mu$ in $D$ is in $\CH_D^Z$ if $\debar\mu=0$, $\bar h \mu=0$ for any $h\in\J_Z$, and $\mu$ has the SEP
with respect to $Z$. Alternatively, by \cite{Aunik}, we have
\begin{equation}\label{CHdef}
\CH_D^{Z}=\{\mu\in\W_D^{Z,\kappa};\, \debar\mu=0\}.
\end{equation}
Notice that $\Hom_{\hol_D}(\Om_D^p,\CH_D^Z)$ can be identified with Coleff--Herrera currents of bidegree $(N-p,\kappa)$
in view of Remark~\ref{gnabbigare}.
Assume as before that there are local coordinates $(z,w)$ such that $Z=\{w=0\}$ and let $\pi(z,w)=z$.
If $\mu\in\Hom_{\hol_D}(\Om_D^p,\CH_D^Z)$,
then $\mu_\alpha(z)$ defined in \eqref{klippning} are $\debar$-closed $(n-p,0)$-currents on $Z$.
Hence, the coefficients $\mu_\alpha(z)$ in the unique representation \eqref{grymta} are in 
$\Om_Z^{n-p}$.

The sheaf $\Ba_Z^{n-p}$ was introduced by Barlet in \cite{Barlet} as the kernel of a certain map
$j_*j^*\Om_Z^{n-p}\to \HH^1_{Z_{sing}}(\Ext^\kappa_{\hol_D}(\hol_Z,\Om_D^{N-p}))$, where $j\colon Z_{reg}\to Z$ is the 
inclusion. It follows from \cite{Barlet} that sections of $\Ba_Z^{n-p}$ are $\debar$-closed meromorphic $(n-p)$-forms on $Z$,
cf.\ \cite[Section~4]{SK} and \cite{HP}. Recall that $\iota\colon Z\to D$ is the inclusion. By \cite[Lemma~4]{Barlet} we have
\begin{equation}\label{kopp}
\iota_*\Ba_Z^{n-p}=\{\mu\in\Hom_{\hol_D}(\Om_D^p,\CH_D^Z);\, \J_Z\mu=d\J_Z\wedge\mu=0\}.
\end{equation}

A current $a$ on $Z$ is \emph{almost semi-meromorphic} if there are a modification $\pi\colon Z'\to Z$ and a semi-meromorphic current
$\nu$ on $Z'$ such that $a=\pi_*\nu$. In particular, $a$ is generically smooth. Thus, if $\mu\in\PM_Z$, then
$a\wedge\mu$ is generically well-defined. By \cite[Theorem~4.8]{AWdirect}, there is a unique $T\in\PM_Z$ such that 
$T=a\wedge\mu$ on the set where $a$ is smooth and $\mathbf{1}_VT=0$, 
where $V$ is the Zariski closure of the singular support of $a$. Henceforth we let
$a\wedge\mu$ denote the extension $T$. One can define
$a\wedge\mu$ as 
\begin{equation}\label{ASMprod}
a\wedge\mu:=\lim_{\epsilon\to 0}\chi(|h|^2/\epsilon)a\wedge\mu,
\end{equation} 
where $h$ is a holomorphic tuple cutting out $V$. 
If $\mu\in\W_Z$, then 
$a\wedge\mu\in\W_Z$.

Let $E_j\to D$, $j=0,\ldots,N$, be complex vector bundles. Let $f_j\colon E_j\to E_{j-1}$ be holomorphic morphisms and
suppose that we have a complex
\begin{equation*}
0\to E_N \stackrel{f_N}{\longrightarrow} \cdots \stackrel{f_1}{\longrightarrow} E_0\to 0,
\end{equation*}
which is exact outside $Z\subset D$. Assume that the associated sheaf complex
\begin{equation}\label{karvkplx}
0\to \hol(E_N) \stackrel{f_N}{\longrightarrow} \cdots \stackrel{f_1}{\longrightarrow} \hol(E_0)
\end{equation}
is exact and let $\F:=\hol(E_0)/\Image \, f_1$ so that \eqref{karvkplx} is a resolution of $\F$.
Recall that $\F$ is Cohen--Macaulay if and only if there is a resolution \eqref{karvkplx} with $N=\kappa$.
Let $Z^\F_j\subset D$ be the set where $f_j$ does not have optimal rank. These \emph{singularity subvarieties} are independent of the resolution,
thus invariants of $\F$, and reflect the complexity of $\F$. It follows from the Buchsbaum-Eisenbud theorem that 
\begin{equation}\label{pinne2}
Z_N^\F\subset \cdots \subset Z_\kappa^\F=Z_{\kappa-1}^\F =\cdots =Z_1^\F=Z
\end{equation} 
and that $\text{codim}_D Z_j^\F\geq j$, $j=\kappa,\kappa+1,\ldots$. Moreover, \cite[Corollary~20.14]{Eisenbud},
$\F$ has pure codimension $\kappa$ (i.e.,
no stalk of $\F$ has embedded primes or associated primes of codimension $>\kappa$) 
if and only if $\text{codim}_D Z_j^\F \geq j+1$ for $j\geq\kappa+1$.

Assume that the $E_j$ are equipped with Hermitian metrics. In this case we say that \eqref{karvkplx} is a
\emph{Hermitian resolution}. Let $\sigma_j\colon E_{j-1}\to E_j$ be the 
Moore-Penrose inverse of $f_j$, i.e., the pointwise
minimal inverse of $f_j$. The $\sigma_j$ are smooth outside $Z$ and almost semi-meromorphic in $D$. Following \cite{AW1},
we define currents $U\in\W_D$ and $R\in\PM_D$ with support in $Z$ and values in $\text{End}\, E$, where $E=\oplus_jE_j$.
Set $\sigma=\sigma_1+\sigma_2+\cdots$ and set $u=\sigma + \sigma\debar\sigma + \sigma(\debar\sigma)^2+\cdots$ outside $Z$. 
Then $fu+u f - \debar f =I_E$, where $f=\oplus_jf_j$. It turns out that $u$ has an almost semi-meromorphic extension to $D$.
In view of \eqref{tabasco} $U$ is given as
\begin{equation*}
U=\lim_{\epsilon\to 0}\chi(|F|^2/\epsilon)u,
\end{equation*}
where $F$ is a (non-trivial) holomorphic tuple vanishing on $Z$. 
Since $fu+u f - \debar f =I_E$,
\begin{equation}\label{tangent}
R:=I_E - (fU+Uf - \debar f) = \lim_{\epsilon\to 0}\debar \chi(|F|^2/\epsilon)\wedge u
\end{equation}
has support in $Z$. It is proved in \cite{AW1} that
\begin{equation*}
R=R_\kappa + R_{\kappa+1} +\cdots,
\end{equation*}
where $R_j\in\PM_D^{0,j}$ has support in $Z$, takes values in $\text{Hom}(E_0,E_j)$, and
\begin{equation}\label{nisse}
fR=\debar R.
\end{equation}
Moreover, if $\varphi\in\hol(E_0)$ then $R\varphi=0$ if and only if $\varphi\in\Image\, f_1$. 

We are interested in the case when $\F$ has pure codimension $\kappa$. It follows from \cite{AW2} that, in this case, 
$R$ has the SEP with respect to $Z$ and that 
\begin{equation}\label{sharp}
R_\kappa \varphi = 0 \iff \varphi\in\Image f_1.
\end{equation}
Let $d\zeta=d\zeta_1\wedge\cdots\wedge d\zeta_N$ for some choice of coordinates $\zeta$ in $D$.
Notice that since $R$ has the SEP, $R\otimes d\zeta$ is a current in $\W_D^{Z,*}$ with values in 
$\text{Hom}(E_0,E)$, i.e.,
\begin{equation}\label{rappt}
R\otimes d\zeta\in\Hom_{\hol_D}(\hol(E_0),E\otimes \W_D^{Z,*}).
\end{equation}

\section{Holomorphic forms, smooth forms, and currents on $X$}\label{holo-smooth-curr}
\subsection{Holomorphic forms on $X$ and associated residue currents}\label{holform}
Recall that the structure sheaf of holomorphic functions on $X$ is defined as $\hol_X=\hol_D/\J$. 
In a similar way one defines the sheaves of K\"{a}hler differentials on $X$.
Let 
\begin{equation*}
\hat\J^p:=\J\cdot\Om^p_D + d\J\wedge \Om^{p-1}_D, \quad 
\Om_{X,\text{K\"{a}hler}}^p:=\Om^p_D/\hat\J^p.
\end{equation*}
Notice that $\hat\J^0=\J$ so that $\Om_{X,\text{K\"{a}hler}}^0=\hol_X$. Notice also that $\Om_{X,\text{K\"{a}hler}}^p$ is an 
$\hol_X$-module.
It is well-known that 
$\Om_{X,\text{K\"{a}hler}}^p$ is intrinsic, i.e., that it does not depend on the embedding of $X$ in $D$. 
Since $\hat\J^0=\J$ has pure dimension it follows that $\Om_{X,\text{K\"{a}hler}}^0=\hol_X$ is torsion-free. 
In general, $\Om_{X,\text{K\"{a}hler}}^p$ has torsion.


The sheaf of \emph{strongly holomorphic} $p$-forms on $X$ is
\begin{equation*}
\Om_X^p:=\Om_{X,\text{K\"{a}hler}}^p/\text{torsion},
\end{equation*}
where torsion means $\hol_X$-torsion.
Notice that $\Om_X^p$ is intrinsic and that it
is the same considered as an $\hol_X$-module or an $\hol_{D}$-module.

\begin{example}
Let $X$ be the subspace of $\C^{2}$ defined by $\mathcal{J}=\langle zw\rangle$. Then $\hol_{X}=\hol_{z}+\hol_{w}$,
where $\hol_z$ and $\hol_w$ are the holomorphic functions only depending on $z$ and $w$, respectively, and 
$d\mathcal{J}=\langle zdw+wdz\rangle$. For the $1$-forms we have 
$\Om^{1}_{X,\text{K\"{a}hler}}=\hol_{X}\{dz,dw\}\big/\langle zdw+wdz\rangle$. If $w\neq 0$, then $\mathcal{J}=\langle z\rangle$ and 
therefore $zdw=wdz=0$. By symmetry this also holds when $z\neq 0$. However, one easily checks that $zdw$ and $wdz$ are not zero as 
K\"{a}hler differentials and therefore they are torsion elements. If we mod these out, the result is a torsion-free module which 
therefore is the strongly holomorphic $1$-forms, i.e.,  $\Om^{1}_{X}=\hol_{z}\{dz\}+\hol_{w}\{dw\}$.
\end{example}

An alternative definition of $\Om_X^p$ is as follows. From a primary decomposition of $\hat\J^p$ one obtains coherent sheaves
$\J^p$ and $\mathscr{S}^p$ such that $\hat\J^p=\J^p\cap\mathscr{S}^p$, $\J^p$ has pure dimension $n$,
and $\mathscr{S}^p$ has dimension $<n$. Hence, $\Om^p_D/\J^p$ has pure dimension and coincides with $\Om_{X,\text{K\"{a}hler}}^p$ 
generically on $Z$. It follows that 
\begin{equation*}
\Om_X^p=\Om^p_D/\J^p.
\end{equation*}
If $X$ is reduced and $j\colon X_{reg}\hookrightarrow D$ is the inclusion, then $\J^p=\{\varphi\in\Om_D^p;\, j^*\varphi=0\}$,
see, e.g., \cite{SK}.

Suppose that $0$ is a smooth point of $Z$ and choose local coordinates $(z,w)$ for $\C^N$ such that $Z=\{w=0\}$. 
Then we can identify $\hol_Z$ with the holomorphic functions of $z$. If $g(z)$ is holomorphic we let 
$\tilde g$ be the extension to ambient space given by $\tilde g(z,w)=g(z)$. In a neighborhood of $0$
we can then define an $\hol_Z$-module structure on $\Om_X^p$ by setting $g\varphi:=\tilde g \varphi$.
Clearly this module structure depends on the choice of local coordinates.   

\begin{proposition}\label{ulv}
Assume that we have coordinates $(z,w)$ so that $Z=\{w=0\}$. Then, with the associated 
$\hol_Z$-module structure, $\Om_X^p$
is coherent. Moreover,
if $\hol_X$ is Cohen--Macaulay, then the following are equivalent
\begin{itemize}
\item[(i)] $\Om_X^p$ is Cohen--Macaulay as an $\hol_X$-module,
\item[(ii)] $\Om_X^p$ is a locally free $\hol_X$-module,
\item[(iii)] $\Om_X^p$ is Cohen--Macaulay as an $\hol_Z$-module, 
\item[(iv)] $\Om_X^p$ is a locally free $\hol_Z$-module.
\end{itemize}
\end{proposition}

\begin{proof}
By the Nullstellensatz there is an $M\in\mathbb{N}$ such that 
$\mathcal{I}:=\langle w^\alpha;\, |\alpha|=M\rangle\subset\J$. Let
$\mathcal{I}^p:=\mathcal{I}\Om^p_D + d\mathcal{I}\wedge \Om^{p-1}_D$ and let
\begin{equation*}
A^p:=\Om^p_D/\mathcal{I}^p.
\end{equation*}
Clearly $A^p$ is coherent  both as an $\hol_{D}$-module and an $\hol_{D}/\mathcal{I}$-module,
and these structures are the same. Moreover, the choice of coordinates makes $A^p$ an $\hol_Z$-module and one checks that
it in fact is a free $\hol_Z$-module. In particular, $A^p$ is a coherent $\hol_Z$-module.
Since $\mathcal{I}\subset\J$ it follows that $\mathcal{I}^p\subset \hat\J^p\subset\J^p$ and so
we have a natural surjective map of $\hol_Z$-modules
\begin{equation*}
A^p\to\Om_X^p,\quad
\varphi + \mathcal{I}^p \mapsto \varphi + \J^p.
\end{equation*}
The kernel $\mathscr{K}$ of this map is $\J^p/\mathcal{I}^p$. Since $\J^p$ is a coherent $\hol_{D}$-module
there are finitely many $\varphi_j\in\J^p$ generating $\J^p$ over $\hol_{D}$. By Taylor expanding any
$g(z,w)\in\hol_{D}$ in the $w$-variables up to order $M$ we see that $\mathscr{K}$ is generated as an $\hol_Z$-module
by $w^\alpha\varphi_j+\mathcal{I}^p$ with $|\alpha|<M$. Since $\mathscr{K}\subset A^p$ and $A^p$ is a coherent $\hol_Z$-module it 
follows that $\mathscr{K}$ is coherent. Hence, $\Om_X^p\simeq A^p/\mathscr{K}$ is a coherent $\hol_Z$-module.

\emph{Claim 1:} $\text{depth}_{\hol_X} \Om_X^p = \text{depth}_{\hol_Z} \Om_X^p$.

\emph{Claim 2:} $n=\text{dim}_{\hol_X} \Om_X^p = \text{dim}_{\hol_Z} \Om_X^p$. 

Recall that, for 
an $R$-module $M$, $\text{dim}_R M:=\text{dim}_R (R/\text{ann}_R M)$ and that $M$ is Cohen--Macaulay if
$\text{depth}_R M=\text{dim}_R M$.

We take these claims for granted for the moment and show that (i), (ii), (iii), and (iv) are equivalent if $\hol_X$ is Cohen--Macaulay.
Notice that it is a local (stalk-wise) statement; in what follows we suppress the point indicating stalk.
If $R$ is a Cohen--Macaulay ring and $M$ is an $R$-module that has a finite free resolution over $R$, then
the Auslander--Buchsbaum formula gives
\begin{equation*}
\text{depth}_R M + \text{pd}_R M = \text{dim}_R R,
\end{equation*}
where $\text{pd}_R M$ is the length of a minimal free resolution of $M$ over $R$, see \cite[Theorem~19.9]{Eisenbud}. 
Thus, $M$ is free over $R$ if and only if $\text{depth}_R M=\text{dim}_R R$.

We now have that $\Om_X^p$ is free over $\hol_X$ if and only if 
$\text{depth}_{\hol_X} \Om_X^p=\text{dim}_{\hol_X} \hol_X$. But 
$\text{dim}_{\hol_X} \hol_X=n=\text{dim}_{\hol_X} \Om_X^p$, where we use the first equality of Claim 1,
so (i) and (ii) are equivalent. In the same way, since $\hol_Z$ is Cohen--Macaulay and $n$-dimensional, (iii) and (iv) are equivalent.
Assume (i) so that $\text{depth}_{\hol_X}\Om_X^p=\text{dim}_{\hol_X}\Om_X^p$. Then by Claims 1 and 2 we get
\begin{equation*}
\text{depth}_{\hol_Z}\Om_X^p = \text{depth}_{\hol_X}\Om_X^p =
\text{dim}_{\hol_X}\Om_X^p = \text{dim}_{\hol_Z}\Om_X^p,
\end{equation*}
and so (iii) follows. In the same way, (iii) implies (i). It remains to prove Claims 1 and 2.

\noindent \emph{Proof of Claim 1:} For notational convenience, let $R=\hol_X$, $R'=\hol_Z$, and $M=\Om_X^p$; notice that $R$ is
a Noetherian local Cohen--Macaulay ring and that $R'$ is a regular Notherian local ring. Since any function in $\J$ vanishes on $Z$
we have an inclusion
$R'\hookrightarrow R$ given by $g(z)\mapsto \tilde g(z,w)+\J$, where $\tilde g(z,w)=g(z)$; cf.\ the $\hol_Z$-module
structure on $ \Om_X^p$. By ``Miracle flatness'', see, e.g., \cite[Corollary~18.17]{Eisenbud} or \cite[Proposition~3.1]{AL}, 
$R$ is a free $R'$ module if and only
if $R$ is Cohen--Macaulay. Thus, $R$ is a free $R'$ module. By \cite[Proposition~18.4]{Eisenbud} and the comment after 
\cite[Corollary~18.5]{Eisenbud},
for a local ring $(A,\mathfrak{m})$ and an $A$-module $N$ we have
\begin{equation*}
\text{depth}_A N=\min \{i;\, \text{Ext}_A^i(A/\mathfrak{m}, N)\neq 0\}.
\end{equation*}
Notice that $R/\mathfrak{m}=\C=R'/\mathfrak{m}$. Claim 1 thus follows if we show that $\text{Ext}_R^i(\C, M)=\text{Ext}_{R'}^i(\C, M)$.
To do this, let 
\begin{equation}\label{hund}
0\to M\to N^0\to N^1\to\cdots
\end{equation}
be a resolution of $M$ as an $R$-module by injective $R$-modules $N^\bullet$. Then
$\text{Ext}_R^i(\C, M)=H^i(\text{Hom}_R(\C,N^\bullet))$.

The complex \eqref{hund} is straightforwardly checked to be exact also considered as a complex of $R'$-modules.
Moreover, by \cite[p.\ 62]{Lam}, since $R$ is a free $R'$-module, the $N^\bullet$ are injective $R'$-modules.
Hence, $\text{Ext}_{R'}^i(\C, M)=H^i(\text{Hom}_{R'}(\C,N^\bullet))$. However,
\begin{equation*}
\text{Hom}_R(\C,N^\bullet) = \text{Hom}_\C(\C,N^\bullet) = \text{Hom}_{R'}(\C,N^\bullet)
\end{equation*}
and so $\text{Ext}_R^i(\C, M)=\text{Ext}_{R'}^i(\C, M)$.

\smallskip

\noindent \emph{Proof of Claim 2:} We know from above that 
\begin{equation*}
\text{dim}_{\hol_X} \Om_X^p= \text{dim}_{\hol_{\C^N}} \Om_X^p=
\text{dim}_{\hol_{\C^N}} \Om_{X,\text{K\"{a}hler}}^p=n.
\end{equation*}
On the other hand, $\text{ann}_{\hol_Z}\Om_X^p=\{0\}$ because if $g(z)\Om^p_{D}\in\J^p$
then $g(z)|_Z=0$. Hence,
\begin{equation*}
\text{dim}_{\hol_Z} \Om_X^p = \text{dim}_{\hol_Z} (\hol_Z/\{0\})=n. 
\end{equation*}
\end{proof}

\begin{corollary}
Assume that $(z,w)$ are coordinates such that $Z=\{w=0\}$, that
$\hol_X$ is Cohen--Macaulay, and that $\Om_X^p$ is Cohen--Macaulay either as an $\hol_X$-module
or as an $\hol_Z$-module. Then, locally there is an $M\in\mathbb{N}$ such that $\Om_X^p$ is generated by
\begin{equation}\label{snub}
\left\{w^\alpha dz^\beta\wedge dw^\gamma + \J^p;\, |\alpha|<M, |\beta|+|\gamma|=p\right\}
\end{equation}
over $\hol_Z$ and a minimal set of generators is an $\hol_Z$-basis.
\end{corollary}

See Example~\ref{matlen} below for a simple illustration of this corollary.

\begin{proof}
Recall the module $A^p$ from the proof of Proposition~\ref{ulv} and let $\varphi(z,w)\in\Om^p_{D}$.
Taylor expanding the coefficients of $\varphi$ with respect to $w$ up to order $M$ shows that $A^p$ is generated as an 
$\hol_Z$-module by \eqref{snub} with $\J^p$ replaced by $\mathcal{I}^p$. Thus, $\Om_X^p$ is generated by
\eqref{snub} over $\hol_Z$. By a standard argument using Nakayama's lemma, a minimal generating set is a basis, cf.\ e.g.,
the proof of \cite[Theorem~2.5]{Mats}.
\end{proof}

\begin{definition}
We let $X_{p\text{-reg}}$ be the subset of $Z_{\text{reg}}$ where $\hol_X$ is Cohen--Macaulay
and $\Om_{X,\text{K\"{a}hler}}^p$ is Cohen--Macaulay.
\end{definition}

\begin{remark}
The property of being Cohen--Macaulay is generic on $Z$ so $X_{p\text{-reg}}$ is a dense open subset of $Z_{reg}$.
Notice also that $\Om_{X,\text{K\"{a}hler}}^p$ is torsion-free where it is Cohen--Macaulay. Hence,
\begin{equation*}
\Om_{X,\text{K\"{a}hler}}^p=\Om_X^p\quad \text{on} \quad X_{p\text{-reg}}.
\end{equation*}
In view of Proposition~\ref{ulv}, thus $\Om_X^p$ and $\Om_{X,\text{K\"{a}hler}}^p$ are locally free $\hol_X$-modules and have locally a structure 
as a free $\hol_Z$-module on $\Xpreg$.
\end{remark}

Assume that \eqref{karvkplx} is a Hermitian resolution of $\Om^p_X$ and that $E_0=T^*_{p,0}D$. If $D$ is pseudoconvex,
such resolutions exist since $\Om^p_X$ is coherent, possibly after replacing $D$ by a slightly smaller set.
Notice that $\hol(E_0)=\Om^p_D$ and that $\Image\, f_1=\J^p$.
For some choice of Hermitian metrics on $E_j$, let $R=R_\kappa + R_{\kappa+1}+\cdots$ be the associated current. 
Since $\Om^p_X$ has pure codimension, by \eqref{rappt} we have
\begin{equation}\label{nyttnr}
\mathcal{R}:=R\otimes d\zeta \in \Hom_{\hol_D}(\Om_D^p,E\otimes\W_D^{Z,*}),
\end{equation}
where $d\zeta=d\zeta_1\wedge\cdots\wedge d\zeta_N$.
Notice that in view of Remark~\ref{gnabbigare}, $\mathcal{R}$ can be identified with an $(N-p,*)$-current
with values in $E$, cf.\ \cite[Section~3]{SK}.
We set $\mathcal{R}_\ell=R_\ell\otimes d\zeta$.


In view of \eqref{nisse} and \eqref{sharp} we have the following lemma.

\begin{lemma}\label{nissen}
The current $\mathcal{R}=\mathcal{R}_\kappa+\mathcal{R}_{\kappa+1}+\cdots$ has bidegree $(N-p,*)$, takes values in $E$,
has the SEP with respect to $Z$, depends only on $d\zeta$ (and $R$), and
\begin{equation*}
f\mathcal{R}=\debar\mathcal{R}.
\end{equation*}
If $\varphi\in\Om^p_D$ then $\mathcal{R}\wedge\varphi = R\wedge\varphi\wedge d\zeta$ and, moreover, $\varphi\in\J^p$
if and only if $\mathcal{R}_\kappa\wedge\varphi=0$.
\end{lemma}

\subsection{Smooth forms on $X$.}
To begin with, in view of Remark~\ref{gnabbigare} we notice that if $\mathcal{I}\subset\Om_D^p$ is a submodule
such that $\J\cdot \Om_D^p\subset \mathcal{I}$,
then we have the isomorphism 
\begin{equation}\label{ekorre}
\{\mu\in \Hom_{\hol_D}(\Om_D^p,\CH_D^Z);\, \mathcal{I}\wedge\mu=0\} \xrightarrow{\sim} \Hom_{\hol_X}(\Om_D^p/\mathcal{I}, \CH_D^{Z}),
\end{equation}
\begin{equation*}
\mu \mapsto (\varphi\mapsto \varphi\wedge \mu).
\end{equation*}
Since $\J^p=\hat\J^p$ generically on $Z$ and any $\mu\in\Hom_{\hol_D}(\Om_D^p,\CH_D^Z)$ 
has the SEP with respect to $Z$, in view of \eqref{ekorre} and Remark~\ref{gnabbigare}
we have 
\begin{eqnarray}\label{koppla}
\Hom_{\hol_X}(\Om_X^p,\CH_D^{Z}) &=& \Hom_{\hol_X}(\Om_{X,\text{K\"{a}hler}}^p,\CH_D^{Z}) \\
&=&
\{\mu\in \Hom_{\hol_D}(\Om_D^p,\CH_D^Z);\, \J\mu=d\J\wedge\mu=0\}. \nonumber
\end{eqnarray}


\begin{definition}\label{smoothdef} Let
\begin{equation*}
\Kernel_p i^*=\{\varphi\in\E^{p,*}_D;\, \varphi\wedge\mu=0, \, \mu\in\Hom_{\hol_X}(\Om_X^p,\CH_D^{Z})\},
\end{equation*}
cf.\ Remark~\ref{gnabbigare}, and let 
\begin{equation*}
\E_X^{p,*}:=\E_D^{p,*}/\Kernel_p i^*
\end{equation*}
be the sheaf of smooth $(p,*)$-forms on $X$.
\end{definition}

Notice that if $\varphi\in \Kernel_p i^*$, then $\debar\varphi\in\Kernel_p i^*$ and so
$\debar$ is well-defined on $\E_X^{p,*}$. 
We write $i^*$ for the natural map $\E_D^{p,*}\to \E_X^{p,*}$.
Notice that if $X$ is reduced, then, in view of \eqref{kopp}
and \eqref{koppla},
a smooth $(p,*)$-form $\varphi$ is in $\Kernel_p i^*$ if and only if $i^*\varphi=0$.
As in \cite[Section~4]{AL} one shows that $\E_X^{p,*}$ is intrinsic, i.e., does not dependent on the
embedding $i\colon X\to D$. 

\begin{proposition}
If $\varphi\in\E_D^{p,*}$ and $\varphi'\in\E_D^{p',*}$ then $i^*(\varphi\wedge\varphi')$ only depends on
$i^*\varphi$ and $i^*\varphi'$. Setting $i^*\varphi\wedge i^*\varphi':=i^*(\varphi\wedge\varphi')$,
$\E_X^{*,*}$ becomes a (bigraded) algebra. In particular, $\E_X^{p,*}$ is an $\E_X^{0,*}$-module.
\end{proposition}

\begin{proof}
Assume that $\varphi\in \Kernel_p i^*$. We must show that $\varphi\wedge\varphi'\in\Kernel_{p+p'}i^*$.
Suppose that $\mu\in\Hom_{\hol_X}(\Om_X^{p+p'},\CH_D^Z)$. Then $\varphi'\wedge\mu$ is 
(a sum of terms) of the form $\xi\wedge\nu$, where $\xi\in\E_D^{0,*}$ and 
$\nu\in\Hom_{\hol_X}(\Om_X^{p},\CH_D^Z)$. Since, by definition of $\Kernel_p i^*$, $\varphi\wedge\nu=0$, it follows that
$\varphi\wedge\varphi'\wedge\mu=0$, and hence $\varphi\wedge\varphi'\in\Kernel_{p+p'}i^*$.
\end{proof}

\begin{proposition}\label{CHvsR}
Let $\mathcal{R}=\mathcal{R}_\kappa+\mathcal{R}_{\kappa+1}+\cdots$ be the residue current associated with 
$\Om^p_X$ defined in Section~\ref{holform} and let $\varphi\in \E_D^{p,*}$. 
Then $\varphi\in\Kernel_p i^*$ if and only if $\mathcal{R}_\kappa\wedge\varphi=0$. 
\end{proposition}
\begin{proof}
Recall the complex \eqref{karvkplx} that we used to define $R$ and therefore also $\mathcal{R}$. 
Consider the dual complex
\begin{equation*}
\cdots \stackrel{f^*_{\kappa+1}}{\longleftarrow}\hol(E^*_\kappa) \stackrel{f^*_{\kappa}}{\longleftarrow}
\cdots \stackrel{f^*_{1}}{\longleftarrow} \hol(E^*_0) \leftarrow 0,
\end{equation*}
where $f_j^*$ is the transpose of $f_j$. 
If $\xi\in \hol(E^*_\kappa)$ and 
$f_{\kappa+1}^*\xi=0$, then, in view of Lemma~\ref{nissen},
\begin{equation}\label{fika}
\debar(\xi\cdot \mathcal{R}_\kappa) = \xi\cdot \debar\mathcal{R}_\kappa
=\xi\cdot f_{\kappa+1}\mathcal{R}_\kappa = f^*_{\kappa+1}\xi\cdot \mathcal{R}_\kappa = 0.
\end{equation}
Hence, $\xi\cdot \mathcal{R}_\kappa$ is a $\debar$-closed (scalar-valued) pseudomeromorphic $(N-p,\kappa)$-current with support on $Z$.
Moreover, since $\J^p\wedge \mathcal{R}=0$, it follows that $\xi\cdot \mathcal{R}_\kappa\in \Hom_{\hol_X}(\Om_X^p,\CH_D^{Z})$.
If $\xi=f^*_\kappa \xi'$, since $\mathcal{R}_{\kappa-1}=0$,
a computation similar to \eqref{fika} shows that $\xi\cdot \mathcal{R}_\kappa=0$.
Hence, we have a map
\begin{equation}\label{bibblan}
\HH^\kappa(\hol(E^*_\bullet), f^*_\bullet) \to \Hom_{\hol_X}(\Om_X^p,\CH_D^{Z}),\quad
[\xi]\mapsto \xi\cdot \mathcal{R}_\kappa.
\end{equation}
By \cite[Theorem~1.5]{ANoether}, this map is an isomorphism.
If $\mathcal{R}_\kappa\wedge\varphi=0$ thus $\varphi\in\Kernel_p i^*$. 

Conversely, assume that $\varphi\in\Kernel_p i^*$. If $\Om_X^p$ is Cohen--Macaulay and \eqref{karvkplx}
is a resolution of minimal length, i.e., if $E_j=0$ for $j > \kappa$, then $\debar\mathcal{R}_\kappa=0$ and so
$\mathcal{R}_\kappa\in \Hom_{\hol_X}(\Om_X^p,\CH_D^{Z})$. In this case, thus, $\mathcal{R}_\kappa\wedge\varphi=0$. In general,
$\Om_X^p$ is Cohen--Macaulay generically on $Z$ and the minimal resolution is a direct summand 
in any resolution. It follows, cf.\ the proof of \cite[Theorem~1.2]{ANoether}, that $\mathcal{R}_\kappa\wedge\varphi=0$
generically on $Z$. By the SEP it then holds everywhere.
\end{proof}

\begin{corollary}\label{gurka}
There is a natural injective 
map $\Om_X^p\hookrightarrow\E_X^{p,0}$.
\end{corollary}
\begin{proof}
Since $\J^p\subset\Kernel_p i^*$, the inclusion $\Om_D^p\subset \E_D^{p,0}$ induces a map $\Om_X^p\to\E_X^{p,0}$.
By Proposition~\ref{CHvsR}, if $\varphi\in\Kernel_p i^*$, then $\mathcal{R}_\kappa\wedge\varphi=0$ and so
$\Kernel_p i^*\cap \Om_D^p=\J^p$ in view of Lemma~\eqref{nissen}. It follows that $\Om_X^p\to\E_X^{p,0}$ is injective.
\end{proof}

The following result is not necessary for this paper but is included here for possible future reference.
We believe that it is interesting in its own right since it shows that the 
de~Rham operator $d=\partial+\debar$ is well-defined on $\E_X$.

\begin{proposition}
The operator $\partial\colon \mathscr{E}^{p,q}_{D}\rightarrow\mathscr{E}^{p+1,q}_{D}$ 
induces a well-defined operator
$\partial\colon\mathscr{E}^{p,q}_{X}\rightarrow\mathscr{E}^{p+1,q}_{X}$.
\end{proposition}
\begin{proof}
We show first that $\partial\colon \mathscr{C}_D^{N-p-1,\kappa}\to \mathscr{C}_D^{N-p,\kappa}$ induces a well-defined
operator
\begin{equation}\label{partial}
\partial\colon \Hom_{\hol_X}(\Om_X^{p+1},\CH_D^Z) \to \Hom_{\hol_X}(\Om_X^{p},\CH_D^Z),
\end{equation}
cf.\ Remark~\ref{gnabbigare}.
Let $\mu\in\Hom_{\hol_X}(\Om_X^{p+1},\CH_D^Z)$, cf.\ \eqref{koppla}.
By \cite[Theorem~3.7]{AWdirect} we have $\partial\mu\in\Hom_{\hol_D}(\Om_D^p,\mathcal{W}^{Z,\kappa}_{D})$ and we 
certainly have $\debar\partial\mu=0$. Hence, $\partial\mu\in \Hom_{\hol_{D}}(\Om^{p}_{D},\CH^{Z}_{D})$.
Moreover, $\mathcal{J}\partial\mu=\partial(\mathcal{J}\mu)\pm d\mathcal{J}\wedge\mu=0$ 
and $d\mathcal{J}\wedge\partial\mu=-\partial(d\mathcal{J}\wedge\mu)=0$.  
In view of \eqref{koppla} thus $\partial\mu\in\Hom_{\hol_X}(\Om_X^{p},\CH_D^Z)$.

Next we notice that if $\varphi\in\E_D^{p,*}$ annihilates 
$\Hom_{\hol_X}(\Om_X^{p},\CH_D^Z)$, then $\varphi$ annihilates $\Hom_{\hol_X}(\Om_X^{p+1},\CH_D^Z)$.
In fact, if $\mu\in\Hom_{\hol_X}(\Om_X^{p+1},\CH_D^Z)$ then $\alpha\wedge\mu\in\Hom_{\hol_X}(\Om_X^{p},\CH_D^Z)$
for all $\alpha\in\Om_D^1$. Hence, $\varphi\wedge\alpha\wedge\mu=0$ for all such $\alpha$ and so $\varphi\wedge\mu=0$.

Assume now that $\varphi\in\Kernel_p i^*$, cf.\ Definition~\ref{smoothdef}.
Then by what we have just noticed,  $\varphi$ annihilates  $\Hom_{\hol_X}(\Om_X^{p+1},\CH_D^Z)$.
If $\mu\in\Hom_{\hol_X}(\Om_X^{p+1},\CH_D^Z)$, in view of \eqref{partial} we thus get
\begin{equation*}
\partial\varphi\wedge\mu=\partial(\varphi\wedge\mu)\pm\varphi\wedge\partial\mu =0.
\end{equation*} 
Hence, $\partial\varphi\in\Kernel_{p+1}i^{*}$ and we conclude that
$\partial(\Kernel_{p}i^{*})\subset \Kernel_{p+1}i^{*}$, from which the proposition immediately follows.  

\end{proof}

\subsection{Smooth forms on $\Xpreg$}\label{smooth-reg}
Here we give a more concrete description of $\E_X^{p,*}$ on $X_{p\text{-reg}}$. 
Choose local coordinates $(z,w)$ centered at a point in $\Xpreg$ such that $Z=\{w=0\}$. Recall that the local coordinates
induce an $\hol_Z$-module structure on $\Om_X^p$. On $\Xpreg$ we get a sequence of mappings
\begin{equation}\label{horlur}
(\hol_Z)^\nu \stackrel{\sim}{\to} \Om_X^p \stackrel{\sim}{\to}
\Hom_{\hol_X}\big(\Hom_{\hol_X}(\Om_X^p,\CH_D^Z),\CH_D^Z\big)
\to (\CH_D^Z)^m
\to (\Om_Z^n)^{m\tilde M}
\end{equation}
defined as follows. On $\Xpreg$, $\Om_X^p$ is a free $\hol_Z$-module  and the first mapping is the isomorphism
given by an $\hol_Z$-basis $\{b_k\}\subset\Om_D^p$ of $\Om_X^p$. 

The second mapping is 
defined on all of $X$ and is the natural mapping
into a double dual, $\varphi\mapsto (\mu\mapsto \varphi\wedge\mu)$, cf.\ Remark~\ref{gnabbigare}. 
It is injective since if $\varphi\in\Om_D^p$ and
$\varphi\wedge\mu=0$ for all $\mu\in\Hom_{\hol_X}(\Om_X^p,\CH_D^Z)$, then $\varphi\in\J^p$; cf.\
the proof of Corollary~\ref{gurka}. It follows from a fundamental theorem of J.-E.\ Roos that the second mapping in fact is an 
isomorphism if and only if $\Om_X^p$ is $S_2$, cf.\ \cite[Theorem~7.3]{AL} and the discussion following it. On $\Xpreg$, $\Om_X^p$
is Cohen--Macaulay, in particular $S_2$, and thus the second mapping is an isomorphism on $\Xpreg$. 

The third mapping depends on a choice of generators $\mu_1,\ldots,\mu_m$ of $\Hom_{\hol_X}(\Om_X^p,\CH_D^{Z})$.
An element $h$ of the double-$\Hom$ then is mapped to the tuple $(h\wedge\mu_1,\ldots,h\wedge\mu_m)$.

For the fourth mapping we choose $M>0$ such that $w^\alpha\mu_j=0$ for $j=1,\ldots,m$ and $w^\alpha\in\J$ if $|\alpha|\geq M$.
Then a tuple $(\nu_j)_j\in(\CH_D^Z)^m$ is mapped to the $m\widetilde M$-tuple $(\pi_*(w^\alpha\nu_j))_{j,|\alpha|<M}$, 
where $\widetilde M$ is the number of monomials $w^\alpha$ with $|\alpha|<M$ and $\pi$ is the projection
$\pi(z,w)=z$. Since $w^\alpha\nu_j$ are $\debar$-closed of bidegree $(N,\kappa)$ in $D$, $\pi_*(w^\alpha\nu_j)$ are $\debar$-closed
of bidegree $(n,0)$ on $Z$, i.e., holomorphic $n$-forms on $Z$. 

Now all mappings in \eqref{horlur} are explained. We will see, Lemma~\ref{smirnoff},
that the composition of these mappings, denoted by $\widetilde T$ from now on, is injective and $\hol_Z$-linear and thus given by a matrix,
also denoted by $\widetilde T$, 
with $\Om_Z^n$-entries.

\smallskip

In order to describe $\E_X^{p,*}$ on $\Xpreg$ we extend $\widetilde T$ to $\E_X^{p,*}$. First,
as for $\Om_X^p$, cf.\ the paragraph before Proposition~\ref{ulv}, we define a $\E_Z^{0,*}$-module structure on $\E_X^{p,*}$ by
\begin{equation*}
\psi\wedge\varphi:=\pi^*\psi \wedge \varphi, \quad \psi\in\E_Z^{0,*}, \, \varphi\in \E_X^{p,*}.
\end{equation*}
The mapping $\Om_X^p\to (\CH_D^Z)^m$ in \eqref{horlur} extends to the mapping
\begin{equation}\label{map1}
\E^{p,*}_X\to (\W^Z_D)^m,\quad \varphi \mapsto (\varphi\wedge\mu_1,\ldots,\varphi\wedge\mu_m).
\end{equation}
Notice, cf.\ Definition~\ref{smoothdef}, that
if $\varphi\in\E_X^{p,*}$ and $\varphi\wedge\mu_j=0$ for all $j$, then $\varphi=0$. Thus, \eqref{map1}
is injective. The mapping $(\CH_D^Z)^m\to (\Om_Z^n)^{m\tilde M}$ in \eqref{horlur} extends to
\begin{equation}\label{mapny}
(\W_D^Z)^m\to (\W_Z)^{m\widetilde M},\quad
(\tau_j)_j\mapsto \big(\pi_*(w^\alpha\tau_j)\big)_{j,|\alpha|<M}.
\end{equation}
In view of \cite[Proposition~4.1 and (4.3)]{AWreg}
and \cite[Theorem~3.5]{AWdirect}, \eqref{mapny} is injective.
Composing \eqref{map1} and \eqref{mapny} we get the injective map
\begin{equation}\label{mapT}
T\colon \E^{p,*}_X\to (\W_Z^{n,*})^{m\widetilde M},\quad
T\varphi=\big(\pi_*(\varphi \wedge w^\alpha \mu_j)\big)_{|\alpha|<M, j=1,\ldots,m}.
\end{equation}
Notice that the restriction of $T$ to $\Om_X^p$ is (after the identification $\Om_X^p\simeq (\hol_Z)^\nu$)
indeed is the mapping $\widetilde T$.
For future reference we also notice that
\begin{equation}\label{mapTsnok}
\widetilde T\colon (\hol_Z)^\nu\to (\Om_Z^{n})^{m\tilde M}, \quad
(h_k)_{k=1,\ldots,\nu}\mapsto \big(\pi_*(\sum_kh_kb_k \wedge w^\alpha \mu_j)\big)_{|\alpha|<M, j=1,\ldots,m}
\end{equation}
on $X_{p\text{-reg}}$. 

\begin{lemma}\label{smirnoff}
The injective mappings $T$ and $\widetilde T$ are $\E_Z^{0,*}$-linear and $\hol_Z$-linear, respectively.
Any $\varphi\in\E_X^{p,*}$
can be written
\begin{equation}\label{klamp}
\varphi = \sum_{k=1}^\nu \varphi_k\wedge b_k + \Kernel_p i^*, \quad \varphi_k\in\E_Z^{0,*},
\end{equation}
on $X_{p\text{-reg}}$ and $T$ is given by matrix multiplication by $\widetilde T$, i.e.,
$T\varphi = \widetilde T (\varphi_1,\ldots,\varphi_\nu)^t$.
\end{lemma}

\begin{proof}
Let $\psi\in\E_Z^{0,*}$. By definition of $T$ and \eqref{projformel},
\begin{eqnarray*}
T(\psi\wedge\varphi) &=& T(\pi^*\psi\wedge\varphi) 
= \big(\pi_*(\pi^*\psi\wedge\varphi \wedge w^\alpha \mu_j)\big)_{|\alpha|<M, j=1,\ldots,m} \\
&=& \psi\wedge \big(\pi_*(\varphi \wedge w^\alpha \mu_j)\big)_{|\alpha|<M, j=1,\ldots,m}.
\end{eqnarray*} 
Hence, $T$ is $\E_Z^{0,*}$-linear. The same computation shows that $\widetilde T$ is $\hol_Z$-linear
and therefore given by a matrix with elements in $\Om_Z^n$.
Explicitly, since any $\varphi\in\Om_X^p$ can be written
$\varphi=\sum_k\varphi_kb_k + \J^p$ for (unique) $\varphi_k\in\hol_Z$,
\begin{equation}\label{TT}
\widetilde T=
\begin{bmatrix}
\pi_*(w^{\alpha_1}b_1\wedge \mu_1) & \dots & \pi_*(w^{\alpha_1}b_\nu\wedge \mu_1) \\
\vdots & \ddots & \vdots \\
\pi_*(w^{\alpha_{\tilde M}}b_1\wedge \mu_m) & \dots & \pi_*(w^{\alpha_{\tilde M}}b_\nu\wedge \mu_m)
\end{bmatrix}
.
\end{equation}
 
Let $\varphi\in\E_X^{p,*}$ and let $\tilde\varphi\in\E_D^{p,*}$ be any representative.
We can write $\tilde\varphi=\sum_i\tilde\varphi'_i\wedge\tilde\varphi''_i$, where $\tilde\varphi'_i\in\E_D^{0,*}$ and 
$\tilde\varphi''_i\in\Om_D^p$.
Moreover, we write $\tilde\varphi'_i=\phi_i+\psi_i$, where every term of $\phi_i$ contains a factor $d\bar{w}_j$ for some $j$
and no term of $\psi_i$ contains such a factor.
Taylor expanding (the coefficients of) $\psi_i$ with respect to $w$ and $\bar w$ to the order $M$ we get
\begin{equation*}
\psi_i(z,w) = \sum_{|\alpha|<M}\frac{\partial^\alpha \psi_i}{\partial w^\alpha}(z,0) \frac{w^\alpha}{\alpha !}
+\sum_{|\alpha|=M} w^\alpha\tilde\psi_{i,\alpha}
+\mathcal{O}(\bar w),
\end{equation*} 
where $\tilde\psi_{i,\alpha}\in\E_D^{0,*}$ and $\mathcal{O}(\bar w)$ is a sum of terms divisible by some $\bar w_j$.
In view of \eqref{koppla} and \eqref{slurp}, $\phi_i$, $w^\alpha\tilde\psi_{i,\alpha}$, and $\mathcal{O}(\bar w)$ are in 
$\Kernel_p i^*$. Hence,
\begin{equation}\label{kyl}
\tilde\varphi=\sum_{i, |\alpha|<M}
\frac{\partial^\alpha \psi_i}{\partial w^\alpha}(z,0) \frac{w^\alpha}{\alpha !} \wedge\tilde\varphi''_i
+\Kernel_p i^*.
\end{equation}
Since $w^\alpha \tilde\varphi''_i\in\Om_D^p$ there are $\tilde\varphi_{\alpha,i,k}\in\hol_Z$ such that
$w^\alpha \tilde\varphi''_i=\sum_k\tilde\varphi_{\alpha,i,k}(z)b_k + \J^p$, and so \eqref{klamp} follows from
\eqref{kyl}.
By $\E_Z^{0,*}$-linearity, 
\begin{equation*}
T(\varphi_k\wedge b_k)= \varphi_k\wedge T|_{\Om_X^p}b_k=\varphi_k\widetilde T (0,\ldots,1_k,\ldots,0)^t
\end{equation*}
and the last statement of the 
lemma follows.
\end{proof}

Notice that by this lemma, $T$ is a map $\E_X^{p,*}\to (\E_Z^{n,*})^{m\widetilde M}$ on $X_{p\text{-reg}}$.

\begin{proposition}\label{rutschkana}
On $X_{p\text{-reg}}$, $\E_X^{p,*}$ is a free $\E_Z^{0,*}$-module, the representation \eqref{klamp} 
of an element $\varphi\in \E_X^{p,*}$ is unique,
and 
\begin{equation*}
\E_X^{p,*}=\E_D^{p,*}/\big(\J\E_D^{p,*} + d\J\wedge \E_D^{p-1,*} + 
\overline{\J_Z}\E_D^{p,*} + d\overline{\J_Z}\wedge\E_D^{p,*}\big),
\end{equation*}
where $\J_Z=\sqrt{\J}$.
\end{proposition}

\begin{proof}
Notice first that since $\widetilde T$ is injective and $\Om_X^p$ is a free $\hol_Z$-module on $X_{p\text{-reg}}$ it follows that,
generically on $X_{p\text{-reg}}$, $\widetilde T$ is a pointwise injective matrix (times $dz_1\wedge\cdots\wedge dz_n$).
Consider a representation \eqref{klamp} and assume that $\sum_k\varphi_k\wedge b_k\in \Kernel_p i^*$. 
Then $\widetilde T (\varphi_1,\ldots,\varphi_\nu)^t=0$. Since $\widetilde T$ is generically pointwise injective on $X_{p\text{-reg}}$
it follows that $\varphi_j=0$, $j=1,\ldots,\nu$, on $X_{p\text{-reg}}$. Hence, the representation \eqref{klamp} is unique and 
$\E_X^{p,*}$ is a free $\E_X^{0,*}$-module. 

It remains to see that
\begin{equation}\label{ruskigare}
\Kernel_p i^* = \J\E_D^{p,*} + d\J\wedge \E_D^{p-1,*} + 
\overline{\J_Z}\E_D^{p,*} + d\overline{\J_Z}\wedge\E_D^{p,*}
\end{equation}
on $X_{p\text{-reg}}$. Let $\mu\in\Hom_{\hol_X}(\Om_X^p,\CH_D^{Z})$ and assume that $\varphi$ is an element of the right-hand side
of \eqref{ruskigare}.
In view of \eqref{slurp}, the terms of $\varphi$ that belong to $\overline{\J_Z}\E_D^{p,*} + d\overline{\J_Z}\wedge\E_D^{p,*}$
annihilate $\mu$ so we may assume that $\varphi\in \J\E_D^{p,*} + d\J\wedge \E_D^{p-1,*}=\hat{\J}^p\wedge\E_D^{0,*}$.
Write $\varphi$ as (a sum of terms) $\varphi'\wedge\varphi''$, where $\varphi'\in\hat\J^p$ and $\varphi''\in\E_D^{0,*}$.
Since $\J^p\supset \hat\J^p$ we have $\varphi'\wedge\mu=0$. Thus $\varphi\wedge\mu=0$ and so $\varphi\in\Kernel_p i^*$.

Assume that $\varphi\in\Kernel_p i^*$ and write $\varphi$ as (a sum of terms) $\varphi'\wedge\varphi''$, where
$\varphi'\in\Om_D^p$ and $\varphi''\in\E_D^{0,*}$. As in the proof of Lemma~\ref{smirnoff}, 
by a Taylor expansion of (the coefficients of) $\varphi''$
with respect to $w$ up to order $M$, we have
\begin{equation}\label{rusk}
\varphi(z,w)=
\varphi'\wedge\sum_{|\alpha|<M}\frac{\partial^\alpha \varphi''}{\partial w^\alpha}(z,0) \frac{w^\alpha}{\alpha !}
+\varphi'\wedge\sum_{|\alpha|=M} w^\alpha\tilde{\varphi}'_{\alpha}
+\mathcal{O}(\bar w, d\bar w),
\end{equation}
where $\tilde{\varphi}'_{\alpha}\in\E_D^{0,*}$ and $\mathcal{O}(\bar w, d\bar w)$ is a sum of smooth terms 
containing either some $\bar{w}_j$ or $d\bar{w}_j$. The second and the last term on the right-hand side of \eqref{rusk} belong to the 
right-hand side of \eqref{ruskigare}. As in the proof of Lemma~\ref{smirnoff} again, this time by writing $w^\alpha\varphi'\in\Om_D^p$
modulo $\J^p$ as a $\hol_Z$-combination of the $b_k$ on $X_{p\text{-reg}}$,
\begin{equation}\label{ruskk}
\varphi'\wedge\sum_{|\alpha|<M}\frac{\partial^\alpha \varphi''}{\partial w^\alpha}(z,0) \frac{w^\alpha}{\alpha !}
=\sum_{k=1}^\nu \phi_k\wedge b_k + \J^p\wedge\E_Z^{0,*},
\end{equation}
where $\phi_k\in \E_Z^{0,*}$. Since $\hat\J^p=\J^p$ on $X_{p\text{-reg}}$ the last term on the right-hand side belongs to
the right-hand side of \eqref{ruskigare}. The sum $S$ in the right-hand side of \eqref{ruskk} is in $\Kernel_p i^*$ since,
by the proof so far, $\varphi$ and $\varphi-S$ are in $\Kernel_p i^*$. In view of Lemma~\ref{smirnoff}, thus
$\widetilde T (\phi_1,\ldots,\phi_\nu)^t=T S=0$. Since $\widetilde T$ is generically pointwise injective on $X_{p\text{-reg}}$,
$\phi_j=0$ on $X_{p\text{-reg}}$. Hence, the left-hand side of \eqref{ruskk} belongs to the right-hand side of \eqref{ruskigare}.
Thus, all terms in the right-hand side of \eqref{rusk} do too, and so \eqref{ruskigare} follows.
\end{proof}

\subsection{Currents and structure forms on $X$.}
The space of $(n-p,n-q)$-currents on $X$ is the dual of the space of compactly supported sections of $\E_X^{p,q}$,
cf.\ \cite[Section~4.2]{HL}.
The topology on $\E_X^{p,*}=\E_D^{p,*}/\Kernel_p i^*$ is the quotient topology. Notice that
$\Kernel_p i^*$ is a closed subspace of $\E_D^{p,*}$ since it is defined as the annihilator of currents.
It follows that the $(n-p,n-q)$-currents on $X$ can be identified with the $(N-p,N-q)$-currents $\mu$
in $D$ such that $\mu . \varphi=0$ for all $\varphi\in\Kernel_p i^*$ with compact support. 
This holds if and only if $\varphi\wedge\mu=0$ for all $\varphi\in\Kernel_p i^*$ since $\Kernel_p i^*$ is both a right and 
left $\E_D^{0,*}$-submodule of $\E_D^{p,*}$. If $\tau$ is an
$(n-p,n-q)$-current on $X$ we write $i_*\tau$ for the corresponding $(N-p,N-q)$-current in $D$. 
Notice that if $\varphi\in\E_D^{p,*}$, then $\varphi\wedge i_*\tau$ only depends on $i^*\varphi$ and we write
\begin{equation*}
\varphi\wedge i_*\tau=i_*(i^*\varphi\wedge\tau).
\end{equation*}
Since $\debar$ is well-defined on  $\E_X^{p,*}$, $\debar$ is defined on $(n-p,*)$-currents $\tau$ on $X$ by 
$\debar\tau.\varphi=\pm\tau.\debar\mu$ and we have $\debar i_*\tau=i_*\debar\tau$.

\begin{definition}\label{cool}
If $\tau$ is an $(n,n)$-current on $X$ with compact support we let
\begin{equation*}
\int_X\tau:= \tau. i^*1.
\end{equation*}
\end{definition}

Notice that $i^*1$ is a well-defined element in $\E_X^{0,0}$ independent of the local embedding $i\colon X\to D$.
Hence, Definition~\ref{cool} makes sense on any pure-dimensional $X$, not just embedded ones.

Let $\mu\in\Hom_{\hol_D}(\Om_D^p,\W_D^{Z,*})$, cf.\ Remark~\ref{gnabbigare}, and assume that $\J^p\wedge\mu=0$. 
Since $\J^p\supset \hat\J^p$ we have $\hat\J^p\wedge\mu=0$ and so, in view of \eqref{slurp} and \eqref{klamp}, if $\varphi\in\Kernel_p i^*$
we get $\varphi\wedge\mu=0$ on $X_{p\text{-reg}}$. Thus, by the SEP, $\varphi\wedge\mu=0$.
Hence, $\mu=i_*\mu'$ for some $(n-p,*)$-current $\mu'$ on $X$.

\begin{definition}\label{WXdef}
The subsheaf $\W_X^{n-p,*}$ of the sheaf of currents on $X$ is defined by
\begin{equation}\label{ljus}
i_*\W_X^{n-p,*} = \{\mu\in\Hom_{\hol_D}(\Om_D^p,\W_D^{Z,*});\, \J\mu=d\J\wedge\mu=0\}.
\end{equation}
\end{definition}

Notice that, since $\hat\J^p=\J^p$ on $X_{p\text{-reg}}$ and currents in $\W_D^Z$ have the SEP with
respect to $Z$, we have,
cf.\ \eqref{koppla},
\begin{equation*}
i_*\W_X^{n-p,*} = \{\mu\in\Hom_{\hol_D}(\Om_D^p,\W_D^{Z,*});\, \J^p\wedge\mu=0\}.
\end{equation*}

Recall that the current $R$ associated with $\Om_X^p$ has the SEP with respect to $Z$
and that $\J^p\wedge R=0$. By \eqref{nyttnr}, $\mathcal{R}$ has the same properties. 
Therefore, 
there is $\omega\in\W_X^{n-p,*}$ such that
\begin{equation}\label{strukturform}
\mathcal{R}=i_*\omega.
\end{equation}
We say that $\omega$ is an $(n-p)$-\emph{structure form} on $X$.

\begin{definition}\label{barlet}
We let $\Ba_X^{n-p}=\{\tau\in\W_X^{n-p,0};\, \debar\tau=0\}$.
\end{definition}

By Definition~\ref{WXdef}, Remark~\ref{gnabbigare}, \eqref{CHdef}, and \eqref{koppla}, we have
\begin{equation}\label{barlet-bjork}
i_*\Ba_X^{n-p}=
\Hom_{\hol_X}(\Om_X^p,\CH_D^{Z}), 
\end{equation}
cf.\ \eqref{kopp}.

\begin{proposition}\label{absolut}
There is a tuple $\omega_0=(\omega_{01},\ldots,\omega_{0r})$, where $\omega_{0i}\in\Ba_X^{n-p}$, and a tuple 
$a_0=(a_{01},\ldots,a_{0r})$ of $E_\kappa$-valued almost semi-meromorphic $(0,0)$-currents $a_{0i}$ in $D$ such that
$a_0$ is 
smooth outside $Z^p_{\kappa+1}:=Z^{\Om_X^p}_{\kappa+1}$ and 
\begin{equation}\label{trumma}
\mathcal{R}_{\kappa}=a_0\cdot i_*\omega_0.
\end{equation}
Moreover, for $j=1,2,\ldots,n$, there are $\text{Hom}(E_{j-1},E_j)$-valued almost semi-meromorphic $(0,j)$-currents $a_j$ in $D$, smooth outside 
$Z^p_{\kappa+j}:=Z^{\Om_X^p}_{\kappa+j}$, such that
\begin{equation}\label{pinne}
\mathcal{R}_{\kappa+j}=a_j \mathcal{R}_{\kappa+j-1},
\end{equation}
where the product is defined as in \eqref{ASMprod}.
\end{proposition}

\begin{proof}
Since $\Kernel f^*_{\kappa+1}\subset \hol(E^*_\kappa)$ is coherent, in particular finitly generated,
there is a trivial vector bundle $F\to D$ and a morphism $g\colon \hol(E_\kappa) \to \hol(F)$
such that the image of the transpose $g^*\colon \hol(F^*)\to \hol(E^*_\kappa)$ equals $\Kernel f^*_{\kappa+1}$.
Notice that $g f_{\kappa+1}=0$ since $f^*_{\kappa+1} g^*=0$.
As in the proofs of \cite[Proposition~3.3]{AS} and \cite[Proposition~3.2]{Sz}, 
the pointwise minimal (with respect to some choice of metric) inverse, $a_0$, of
$g$ is smooth outside $Z^p_{\kappa+1}$, has an almost semi-meromorphic extension across $Z^p_{\kappa+1}$, and $R_\kappa=a_0gR_\kappa$.
Hence,
\begin{equation}\label{trummor}
\mathcal{R}_\kappa=R_\kappa\otimes d\zeta=a_0 g \mathcal{R}_\kappa.
\end{equation} 
In view of Lemma~\ref{nissen} we have
\begin{equation*}
\debar g \mathcal{R}_\kappa = g f_{\kappa+1} \mathcal{R}_\kappa=0,
\end{equation*}
$g\mathcal{R}_\kappa$ is an $F$-valued section of $\Hom_{\hol_D}(\Om_D^p,\W_D^{Z,*})$, and $\J^p\wedge g\mathcal{R}_\kappa =0$. 
Thus, after a choice of frame of $F$, 
we can identify $g\mathcal{R}_\kappa$ with a tuple $\omega_0$ of sections of $\Ba_X^{n-p}$, i.e.,
$g \mathcal{R}_\kappa=i_*\omega_0$. By the choice of frame of $F$, $a_0$ is a tuple of $E_\kappa$-valued 
almost semi-meromorphic currents. Hence, \eqref{trumma} follows from \eqref{trummor}. 

By \cite[Theorem~4.4]{AW1}, in $D\setminus Z^p_{\kappa+j}$ there are smooth $(0,j)$-forms $a_j$ such that
$R_{\kappa+j}=a_j R_{\kappa+j-1}$.
As in the proof of \cite[Proposition~3.3]{AS} the $a_j$ have almost semi-meromorphic extensions (also denoted $a_j$)
across $Z^p_{\kappa+j}$ and $R_{\kappa+j}=a_j R_{\kappa+j-1}$ 
holds in $D$; here $a_jR_{\kappa+j-1}$ is defined as in \eqref{ASMprod}, and we remark that for this last identity 
to hold in $D$ it 
is necessary that $\Om_X^p$ has pure dimension. Thus, 
\eqref{pinne} follows.
\end{proof}

\section{The sheaf $\mathcal{V}_X^{p,*}$.}\label{Vsektion}
The sheaf $\mathcal{V}_X^{p,*}$ is intrinsic on $X$ and an extension of $\E_X^{p,*}$.
In terms of our local embedding $i\colon X\to D$ the idea is as follows. Recall that $Z=X_{\text{red}}$ and that $\Om_X^p$ locally on 
$\Xpreg\subset Z_{\text{reg}}$
is a free $\hol_Z$-module, where the module structure depends on a choice of local coordinates. 
As in Section~\ref{smooth-reg} we let $\{b_k\}_{k=1}^\nu$ be a local $\hol_Z$-basis of $\Om_X^p$.
By Lemma~\ref{smirnoff}, each $\varphi\in\E_X^{p,*}$ has a representative $\sum_k\varphi_k \wedge b_k$
on $X_{p\text{-reg}}$,
where $\varphi\in\E_Z^{0,*}$. One can define $\mathcal{V}_X^{p,*}$ on $\Xpreg$ as such sums with $\varphi_k\in\W_Z^{0,*}$ instead of 
$\E_X^{0,*}$ and require $\varphi_k$ to transform under changes of coordinates and basis $\{b_k\}$ as in the case of $\E_X^{p,*}$.
However, we choose a more invariant approach. To motivate it we notice that each sum  $\sum_k\varphi_k \wedge b_k$ with $\varphi_k\in\W_Z^{0,*}$
induces an $\hol_X$-linear mapping  $\Ba_X^{n-p}\to\W_X^{n,*}$ as follows.

Let $\mu\in\Ba_X^{n-p}$. Then $b_k\wedge i_*\mu$ is in $\CH_D^{Z}$ and depends only 
on the class of $b_k$ in $\Om_X^p$.
Moreover, $\J b_k\wedge i_*\mu=0$. In view of the representation \eqref{grymta} of $i_*\mu$, if $\varphi_k\in\W_Z^{0,*}$, then 
$\varphi_k\wedge b_k\wedge i_*\mu$
is well-defined in $\W_D^{Z,*}$ since $\varphi_k\wedge\debar(dw/w^{\alpha+\mathbf{1}})$
exists as a tensor product. Moreover, $\J \varphi_k\wedge b_k\wedge i_*\mu=0$ and so  
$\varphi_k\wedge b_k\wedge i_*\mu$ defines an element in $\W_X^{n,*}$. Hence, $\varphi_k\wedge b_k$ induces a mapping
$\Ba_X^{n-p}\to\W_X^{n,*}$.


\smallskip

With this in mind we make the following definition.
\begin{definition}
$\mathcal{V}_X^{p,*}:=\Hom_{\hol_X}(\Ba_X^{n-p},\W_X^{n,*})$.
\end{definition}
\begin{remark}
The sheaf $\V_X^{0,*}$ was introduced in \cite[Section~7]{AL} but was denoted by $\W_X^{0,*}$ there. In this paper
$\W_X^{0,*}$ naturally has another meaning, see Definition~\ref{WXdef}; cf.\ also Proposition~\ref{Vprop} below.  
\end{remark}
If $\varphi\in\E_D^{p,*}$, then $\varphi$ defines an element $\varphi'$ in $\mathcal{V}_X^{p,*}$ by
$\varphi'(\mu)=\tau$, where $i_*\tau=\varphi\wedge i_*\mu$.
By Definition~\ref{smoothdef} and \eqref{barlet-bjork}, $\varphi'=0$ if and only if $\varphi\in\Kernel_p i^*$.
Hence, we have a well-defined injection
\begin{equation*}
\E_X^{p,*}\hookrightarrow \mathcal{V}_X^{p,*}.
\end{equation*}

In consistency with Remark~\ref{gnabbigare}, for $\varphi\in\V_X^{p,*}$ and $\mu\in\Ba_X^{n-p}$
we write $\varphi\wedge\mu$ instead of $\varphi(\mu)$.

\begin{definition}
Let $\varphi, \psi\in\V_X^{p,*}$. We say that $\debar\varphi=\psi$ if $\debar (\varphi\wedge\mu)=\psi \wedge\mu$
for all $\mu\in\Ba_X^{n-p}$.
\end{definition}


\begin{proposition}\label{hard}
Let $\varphi\in\mathcal{V}_X^{p,*}$. On $X_{p\text{-reg}}$ there are $\varphi_k\in\W_Z^{0,*}$ such that, for
any $\mu\in\Ba_X^{n-p}$,
\begin{equation}\label{ulan}
i_*\varphi\wedge\mu=\sum_{k=1}^\nu\varphi_k\wedge b_k\wedge i_*\mu
=\sum_{k=1}^\nu\sum_\alpha \varphi_k\wedge \pi_*(w^\alpha b_k\wedge i_*\mu)\wedge\debar\frac{dw}{w^{\alpha+\mathbf{1}}}.
\end{equation}
\end{proposition}

\begin{proof}
Recall from \eqref{mapTsnok} the matrix $\widetilde{T}$. We can choose a holomorphic matrix $\widetilde A$ such that
\begin{equation}\label{blurp}
(\hol_Z)^\nu \stackrel{\widetilde{T}}{\longrightarrow} (\Om_Z^n)^{m\widetilde{M}}
\stackrel{\widetilde{A}}{\longrightarrow} (\Om_Z^n)^{M'}
\end{equation}
is exact. Then also 
\begin{equation}\label{slurp2}
(\W_Z^{0,*})^\nu \stackrel{\widetilde{T}}{\longrightarrow} (\W_Z^{n,*})^{m\widetilde{M}}
\stackrel{\widetilde{A}}{\longrightarrow} (\W_Z^{n,*})^{M'}
\end{equation}
is exact. To see this, notice first that \eqref{blurp} is generically pointwise exact.
Take Hermitian metrics on the vector bundles underlying the free sheaves in \eqref{blurp}
and let $\widetilde B$ and $\widetilde S$ be the Moore-Penrose inverses of $\widetilde T$ and $\widetilde A$, respectively.
Then $\widetilde B$ and $\widetilde S$ are almost semi-meromorphic, cf.\ the definition of $\sigma_j$ in connection to 
\eqref{karvkplx}. Moreover, on the set where \eqref{blurp} is pointwise exact, $\widetilde{S}\widetilde{A}+\widetilde{T}\widetilde{B}$ 
is the identity on $(\Om_Z^n)^{m\widetilde{M}}$.
Thus, if $\mu\in(\W_Z^{n,*})^{m\widetilde{M}}$ and 
$\widetilde{A}\mu=0$, we have $\mu=\widetilde{T}\widetilde{B}\mu$ since $\W$ 
is closed under multiplication by almost semi-meromorphic currents, cf.\ \eqref{ASMprod}.

Let $\varphi\in \mathcal{V}_X^{p,*}$ and let $\mu_j$, $j=1,\ldots,m$, be generators of $\Ba_X^{n-p}$. For notational convenience,
we will identify 
$\varphi\wedge\mu_j$ and $i_*\varphi\wedge\mu_j$ as well as $\mu_j$ and the corresponding currents in $i_*\Ba_X^{n-p}$. 
In view of \eqref{grymta},
\begin{equation}\label{bator}
\varphi\wedge\mu_j=\frac{1}{(2\pi i)^\kappa}\sum_\alpha \pi_*(w^\alpha\varphi\wedge\mu_j)\wedge\debar\frac{dw}{w^{\alpha+\mathbf{1}}}.
\end{equation} 
We claim that the tuple $(\pi_*(w^\alpha\varphi\wedge\mu_j))_{\alpha,j} \in (\W_Z^{n,*})^{m\widetilde{M}}$ is in the image of 
$(\W_Z^{0,*})^\nu$ under
$\widetilde{T}$. Given the claim, there are $\varphi_k\in\W_Z^{0,*}$ such that, cf.\ \eqref{TT},
\begin{equation*}
\pi_*(w^\alpha\varphi\wedge\mu_j) = (2\pi i)^\kappa\sum_k\varphi_k\wedge \pi_*(w^\alpha b_k\wedge\mu_j).
\end{equation*}
By \eqref{bator}, \eqref{ulan} follows with $\mu=\mu_j$. Since $\mu_j$ generate $\Ba_X^{n-p}$, \eqref{ulan} follows.

It remains to prove the claim. By exactness of \eqref{slurp2} we need to show that 
\begin{equation}\label{snubbe}
\widetilde{A} (\pi_*(w^\alpha\varphi\wedge\mu_j))_{\alpha,j}=0.
\end{equation}
In view of Proposition~\ref{AWprop} it is enough to show \eqref{snubbe} where $\pi_*(w^\alpha\varphi\wedge\mu_j)$ are smooth
and \eqref{blurp} is pointwise exact. Fix such a point; for notational convenience, suppose it is $0$.

Let \eqref{karvkplx} be a minimal Hermitian resolution of $\Om_{X}^p$ in a neighborhood of $0$. Since 
$\Om_{X}^p$ is Cohen--Macaulay on $\Xpreg$, $E_\ell=0$ for $\ell>\kappa$, and the corresponding 
currents $R=R_\kappa$ and  $\mathcal{R}=\mathcal{R}_\kappa$ are $\debar$-closed. Since the mapping \eqref{bibblan} is an isomorphism it follows
that the components, $\mu_j$, $j=1,\ldots,m$, of $\mathcal{R}$ (with respect to some frame of $E_\kappa$)  generate $i_*\Ba_X^{n-p}$. 
Let $(\hol(E'_\bullet),f'_\bullet)$ be the Koszul complex of the regular sequence $z_1,\ldots,z_n$ in $D$. 
Then $(\hol(E'_\bullet),f'_\bullet)$ equipped with the trivial metric is a Hermitian resolution of $\hol_D/\langle z\rangle$, $\hol(E'_0)=\hol_D$, 
$\hol(E'_n)=\hol_D$, and  the corresponding current is $R'=\debar (1/z):=\debar(1/z_1)\wedge\cdots\wedge\debar(1/z_n)$.

Let $(\hol(E''_\bullet),f''_\bullet)$ be the tensor product of the complexes $(\hol(E_\bullet),f_\bullet)$ and $(\hol(E'_\bullet),f'_\bullet)$,
i.e., $E''_k=\oplus_{i+j=k}E_i\otimes E'_j$ and $f''_\bullet=f_\bullet\otimes \mathbf{1}_{E'} + \mathbf{1}_{E}\otimes f'_\bullet$.
As the tensor product of minimal resolutions of properly intersecting Cohen--Macaulay modules, $(\hol(E''_\bullet),f''_\bullet)$ is a resolution
of $\mathscr{F}:=\hol(E''_0)/\Image f''_1$. Notice that $\hol(E''_0)=\hol(E_0)\otimes \hol(E'_0)=\hol(E_0)=\Om^p_D$ and that
$\mathcal{I}:=\Image f''_1=\Image f_1 \cdot \hol(E'_0)+ \Image f'_1 \cdot \hol(E_0)$ so that 
\begin{equation}\label{matta}
\mathscr{F}=\Om_D^p/\mathcal{I}=\Om_D^p/(\J^p + \langle z\rangle \Om_D^p).
\end{equation}
Clearly $\mathscr{F}$ is supported at $0$ and since $(\hol(E''_\bullet),f''_\bullet)$ has length $\kappa+n=N$, 
$\mathscr{F}$ is Cohen--Macaulay and $(\hol(E''_\bullet),f''_\bullet)$ is a minimal resolution. Following
\cite[Section~4]{Akrit}, the product $R\wedge R'$ makes sense and is the current $R''$ associated with
$(\hol(E''_\bullet),f''_\bullet)$. It follows that $\mu_j\wedge\debar(1/z)$ generate $\Hom_{\hol_X}(\Om_D^p/\mathcal{I},\CH_D^{Z})$.
Since $\mathscr{F}$ is Cohen--Macaulay, for the same reason that the second map in \eqref{horlur} is an isomorphism on $\Xpreg$, the map
\begin{equation}\label{korvskinn}
\mathscr{F}\to \Hom_{\hol_X}\big(\Hom_{\hol_X}(\Om_D^p/\mathcal{I},\CH_D^{Z}), 
\Hom_{\hol_X}(\hol_D/(\J+\langle z\rangle),\CH_D^{Z})\big),  
\end{equation}
\begin{equation*}
\phi\mapsto (\mu\mapsto \phi\wedge\mu),
\end{equation*}
is an isomorphism. In view of \eqref{grymta}, if $\phi\in\mathscr{F}$, then
\begin{equation}\label{kola}
\phi\wedge \mu_j\wedge \debar(1/z)=\frac{1}{(2\pi i)^N}\sum_\alpha \pi'_*(w^\alpha \phi\wedge \mu_j\wedge \debar\frac{1}{z})
\wedge\debar\frac{dw}{w^{\alpha+\mathbf{1}}}\wedge\debar\frac{dz}{z},
\end{equation}
where $\pi'$ is the map $(z,w)\mapsto 0$. The tuple 
$(\pi'_*(w^\alpha \phi\wedge \mu_j\wedge \debar(1/z)))_{\alpha,j}\in \C^{m\widetilde{M}}$
determines $\phi$ and we have the injective map
\begin{equation}\label{gullviva}
\mathscr{F}\to \C^{m\widetilde{M}},
\end{equation}
cf.\ \eqref{mapTsnok}.
In view of \eqref{matta}, since $b_k$ generate $\Om_X^p=\Om_D^p/\J^p$ over $\hol_Z$, $b_k$ also generate
$\mathscr{F}$ over $\hol_Z$. Hence, we have the surjective map $(\hol_Z)^\nu\to \mathscr{F}$, $(h_k)_k\mapsto \sum_kh_kb_k$.
Composing with \eqref{gullviva}, we get 
\begin{equation*}
\widetilde{\mathcal{T}}\colon (\hol_Z)^\nu\to \C^{m\widetilde{M}},\quad
\widetilde{\mathcal{T}}(h_k)_k=\big(\sum_k\pi'_*(w^\alpha h_kb_k\wedge \mu_j\wedge \debar\frac{1}{z})\big)_{\alpha,j}.
\end{equation*}
Recall (again) the map $\widetilde{T}$ from \eqref{mapTsnok} and \eqref{TT} and write $\widetilde{T}=\widetilde{T}'dz$, where 
$\widetilde{T}'$ is a matrix of holomorphic functions. Let $\pi''\colon Z\to \{0\}$ and notice that $\pi'=\pi''\circ \pi$.
We get
\begin{eqnarray*}
\pi'_*(w^\alpha h_kb_k\wedge \mu_j\wedge \debar\frac{1}{z}) &=& 
\pi''_*\big(\pi_*(w^\alpha b_k\wedge \mu_j)h_k\wedge \debar\frac{1}{z}\big) =
\pi''_*\big(\widetilde{T}'_{\alpha,j,k} h_k dz\wedge \debar\frac{1}{z}\big) \\
&=&
\widetilde{T}'_{\alpha,j,k}(0)h_k(0).
\end{eqnarray*}
Hence, $\widetilde{\mathcal{T}}dz=\widetilde{T}(0)$. Since \eqref{blurp} is pointwise exact at $0$
it follows that a tuple $(\lambda_{\alpha,j})\in\C^{m\widetilde{M}}$ is in the image of $\widetilde{\mathcal{T}}$
if and only if $\widetilde A (0) (\lambda_{\alpha,j})=0$. We remark that this implies that $\widetilde T$ is pointwise
injective on $\Xpreg$.

Now, by \eqref{bator}, since $\pi_*(w^\alpha\varphi\wedge\mu_j)$ is smooth in a neighborhood of $0$, we have
\begin{equation}\label{bator2}
\varphi\wedge\mu_j=\sum_\alpha\sum'_{|L|=*} \phi_{j,\alpha,L}(z)\wedge d\bar{z}_L\wedge\debar\frac{dw}{w^{\alpha+\mathbf{1}}},
\end{equation}
where $\phi_{j,\alpha,L}(z)$ are smooth $(n,0)$-forms on $Z$ and $d\bar{z}_L$ are a basis of $T^*_{0,*}Z$.
Set 
\begin{equation}\label{sas}
\phi_L(\mu_j\wedge \debar(1/z)):=\sum_\alpha \phi_{j,\alpha,L}(z)\wedge\debar\frac{dw}{w^{\alpha+\mathbf{1}}}
\wedge \debar\frac{1}{z}.
\end{equation} 
To see that $\phi_L$ is well-defined, recall that $\mu_j$ are the components of $\mathcal{R}$.
Since \eqref{bibblan} is an isomorphism it follows that the relations between the $\mu_j$ are generated by $f_\kappa^*$.
In the same way, it follows that the relations between $\mu_j \wedge \debar(1/z)$, which are the components of $R''$,
are generated by $(f_{\kappa+n}'')^* = f_\kappa^*\otimes \mathbf{1}_{(E'_n)^*} \oplus \mathbf{1}_{E_\kappa^*}\otimes (f'_n)^*$.
Thus, if $a_j$ are such that $\sum_j a_j'' \mu_j \wedge \debar(1/z) = 0$, then we have that $a_j'' = a_j + a_j'$, where
$\sum_j a_j \mu_j = 0$, and $a_j' \debar(1/z) = 0$. This implies that $\phi_L$ is well-defined.

Now, $\phi_L(\mu_j\wedge \debar(1/z))$ is an $(N,N)$-current, in particular $\debar$-closed, and it is annihilated by
$\langle z \rangle$. Moreover, it is annihilated by $\J$ since $\varphi\wedge\mu_j$ is, and
\begin{equation*}
\varphi\wedge\mu_j\wedge d\bar{z}_{L^c}=\pm d\bar{z}\wedge\sum_\alpha \phi_{j,\alpha,L}(z)\wedge\debar\frac{dw}{w^{\alpha+\mathbf{1}}},
\end{equation*}
where $L^c=\{1,\ldots,n\}\setminus L$. Hence, $\phi_L(\mu_j\wedge \debar(1/z))$ is in $\Hom_{\hol_X}(\hol_D/(\J+\langle z\rangle),\CH_D^{Z})$
and it follows that $\phi_L$ is in the right-hand side of \eqref{korvskinn}.
Since  \eqref{korvskinn} is an isomorphism, $\phi_L$ is multiplication by an element, also denoted $\phi_L$, in $\mathscr{F}$.
In view of \eqref{kola} and \eqref{sas}, the image under \eqref{gullviva} of $\phi_L$ is the tuple
\begin{equation*}
(2\pi i)^N \big(\phi_{j,\alpha,L}(0)\big)_{\alpha,j}.
\end{equation*}
It is in the image of $\widetilde{\mathcal{T}}$ and hence in the kernel of $\widetilde A (0)$. Thus,
\begin{equation*}
0=\widetilde A (0) \big(\sum'_{|L|=*}\phi_{j,\alpha,L}(0)\wedge d\bar{z}_L\big)_{\alpha,j}.
\end{equation*}
However, in view of \eqref{bator} and \eqref{bator2}, $\sum'_{|L|=*}\phi_{j,\alpha,L}(0)\wedge d\bar{z}_L$ is the value 
of $\pi_*(w^\alpha\varphi\wedge\mu_j)$ at $0$ and so \eqref{snubbe} follows at $0$.
Hence, \eqref{snubbe} follows at points where $\pi_*(w^\alpha\varphi\wedge\mu_j)$ are smooth
and \eqref{blurp} is pointwise exact, concluding the proof of the claim.
\end{proof}

By Proposition~\ref{hard}, if $\varphi\in\VV_X^{p,*}$ then, on $\Xpreg$, there are $\varphi_k\in\W_Z^{0,*}$ such that 
$\varphi$ is given by multiplication by $\sum_k\varphi_k\wedge b_k$ in the way described in the second paragraph of this section. 
In this way we can identify $\VV_X^{p,*}$ with such sums
on $\Xpreg$. 

The following lemma is proved in the same way as Lemma~7.7 and Corollary~7.8 are proved in \cite{AL}.

\begin{lemma}\label{rutat}
Each $\varphi\in\VV_X^{p,*}=\Hom_{\hol_X}(\Ba_X^{n-p},\W_X^{n,*})$ has a unique extension to an element in
$\Hom_{\E_X^{0,*}}(\E_X^{0,*}\wedge\Ba_X^{n-p},\W_X^{n,*})$.
Moreover, if $\mu\in\W_X^{n-p,*}$ is such that $i_*\mu=\sum_\ell a_\ell\wedge i_*\mu_\ell$, where $\mu_\ell\in\Ba_X^{n-p}$ and
$a_\ell$ are almost semi-meromorphic in $D$ and generically smooth on $Z$, then $\varphi\wedge\mu$ is well-defined in  
$\W_X^{n,*}$ by the formula
\begin{equation*}
i_* (\varphi\wedge\mu)=\sum_\ell (-1)^{\text{deg}\, a_\ell\cdot \text{deg}\, \varphi} a_\ell\wedge i_*(\varphi\wedge\mu_\ell),
\end{equation*}
where the products by $a_\ell$ are defined as in \eqref{ASMprod}.
\end{lemma}

By this lemma $\V_X^{p,*}$ gets a natural $\E_X^{0,*}$-module structure,
which is the same as the $\E_X^{0,*}$-module structure it inherits from $\W_X^{n,*}$.

\subsection{The sheaf $\V_X^{p,*}$ in case $X$ is reduced}\label{Vred}

\begin{proposition}\label{Vprop}
If $X=Z$ is reduced then $\mathcal{V}_X^{p,*}=\W_{X}^{p,*}$. 
\end{proposition}

\begin{lemma}\label{lemmaW1}
If $\pi\colon\widetilde{Z}\rightarrow Z$ is a modification then $\pi_{*}\colon\W_{\widetilde{Z}}\rightarrow\W_{Z}$ is a bijection.
\end{lemma}

\begin{proof}
Denote the exceptional set of the modification by $E$. If $\pi_{*}\tau=0$ then $\tau$ is zero on $\widetilde{Z}\setminus E$ 
and by the SEP $\tau$ is zero everywhere. Hence $\pi_{*}$ is injective.

To show that the map is surjective pick $\nu\in\W_{Z}$. By \cite[Proposition~1.2]{Alitennot} there is a 
$\widetilde{\tau}\in\PM_{\widetilde{Z}}$ such that $\pi_{*}\widetilde{\tau}=\nu$ . We have 
$\widetilde{\tau}\in\W_{\widetilde{Z}\setminus E}$ since $\pi$ is a biholomorphism on $\widetilde{Z}\setminus E$. 
If we let $\tau:=\mathbf{1}_{\widetilde{Z}\setminus E}\widetilde{\tau}$ then $\tau\in\W_{\widetilde{Z}}$ since $\tau$ 
must have the SEP with respect to every subvariety. We also have $\pi_{*}\tau=\nu$ since both 
$\pi_{*}\tau$ and $\pi_{*}\widetilde{\tau}$ are in $\W_{Z}$ and they are equal generically and therefore equal everywhere.
\end{proof}

\begin{lemma}\label{lemmaW2}
Given $\nu\in\W^{n,q}_{Z}$ and a generically non-zero $\mu\in\Ba^{n}_{Z}$ there is a unique 
$\nu'\in\W^{0,q}_{Z}$ such that $\nu=\mu\wedge\nu'$.
\end{lemma}

\begin{proof}
Let $\pi:\widetilde{Z}\rightarrow Z$ be a resolution of singularities. Then $\pi^*\mu$ is a generically non-zero meromorphic
$n$-form on $\widetilde Z$. Moreover, by Lemma~\ref{lemmaW1} there is 
a unique $\tau\in\W^{n,q}_{\widetilde{Z}}$ such that $\pi_{*}\tau=\nu$. In view of \cite[Theorem~3.7]{AWdirect}, 
since $\widetilde{Z}$ is smooth, $\tau$ is a 
$K_{\widetilde Z}$-valued section of $\W^{0,q}_{\widetilde{Z}}$. Thus, $\tau':=\tau/\pi^*\mu$ is a section
of $\W_{\widetilde Z}^{0,q}$, and $\tau=\pi^{*}\mu\wedge\tau'$, cf.\ \eqref{ASMprod}. 
Then $\nu=\pi_{*}\tau=\pi_{*}(\pi^{*}\mu\wedge\tau')=\mu\wedge\pi_{*}\tau'$ and thus $\pi_{*}\tau'$ does the job.

If we have two currents satisfying the lemma then they are equal where $\mu$ is non-zero. 
By assumption this means that they are equal generically and then, by the SEP, they are equal everywhere.
\end{proof}

\begin{remark}\label{sladd}
Any $h\in\Hom_{\hol_Z}(\Ba_{Z}^{n-p},\W^{n,q}_{Z})$ naturally extends to operate on forms 
$f\mu$, where $f$ is a germ of a meromorphic function on $Z$, and $\mu\in\Ba_{Z}^{n-p}$. 
The extension is unique and $h$ becomes linear over the sheaf of meromorphic functions on $Z$. 
Notice that $f\mu$ is not necessarily in $\Ba_{Z}^{n-p}$.
\end{remark}

\begin{proof}[Proof of Proposition~\ref{Vprop}.]
The currents in $\Ba_{Z}^{n-p}$ are meromorphic and in particular almost semi-meromorphic. 
In view of \eqref{ASMprod} and the comment following it, $a\wedge\nu$ is well-defined and in $\W_Z$ for any 
almost semi-meromorphic current $a$ on $Z$ and any $\nu\in \W_Z$.
Hence we can define a map $\Psi\colon\W_{Z}^{p,q}\rightarrow\Hom_{\hol_Z}(\Ba_{Z}^{n-p},\W^{n,q}_{Z})$ by $(\Psi\nu)(\mu)=\mu\wedge\nu$. 
If $\mu\wedge\nu=0$ for all $\mu\in\Ba_{Z}^{n-p}$ then $\nu=0$ on $Z_{reg}$. But then, by the SEP,
$\nu=0$ on $Z$ and hence $\Psi$ is injective.

To show that $\Psi$ is surjective take $h\in\Hom_{\hol_Z}(\Ba_{Z}^{n-p},\W^{n,q}_{Z})$. Suppose we have a local 
parametrization $Z\cap(\Delta_{z}\times\Delta_{w})\rightarrow\Delta_{z}$ of $Z$, where $\Delta_z$ and $\Delta_w$ are polydiscs in $\C^{n}_z$
and $\C^\kappa_w$, respectively,   so that $\{d z_{I}\}_{|I|=n-p}$ 
generically is a basis for $\Ba_{Z}^{n-p}$. This means that $\mu\in\Ba^{n-p}_{Z}$ may be written $\mu=\sum_{|I|=n-p}f_{I}d z_{I}$ 
for some meromorphic functions $f_{I}$ on $Z$. Therefore, by Remark~\ref{sladd}, it suffices to find $\nu\in\W_{Z}^{p,q}$ so that 
$h(d z_{I})=d z_{I}\wedge\nu$ for all $I$. By Lemma \ref{lemmaW2} there are unique $\nu_{J}\in\W^{0,q}_{Z}$ with $h(d z_{J})=d z\wedge\nu_{J}$. 
We let $\nu=\sum_{J}d z_{J^{c}}\wedge\nu_{J}$, so that $\nu\in\W^{p,q}_{Z}$, and get 
$d z_{I}\wedge\nu=\sum_{J}d z_{I}\wedge d z_{J^{c}}\wedge\nu_{J}=d z\wedge\nu_{I}=h(d z_{I})$.
\end{proof}

\section{Integral operators on $X$}\label{intop}
Given our local embedding $i\colon X\to D\subset\C^N$ as usual and a choice of local coordinates $z$ in $D$ we define 
integral operators and prove their basic mapping properties.

Let $R$ and $\mathcal{R}$ be the currents associated with a Hermitian resolution \eqref{karvkplx} of $\Om_X^p$
such that $E_0=T^*_{p,0}D$. The (full) Bochner-Martinelli form in $D_\zeta\times D_z$, where $\zeta$ and $z$ 
are the same local coordinates in $D$, is
\begin{equation*}
B=\sum_{j=1}^N \frac{1}{(2\pi i)^j}\frac{\partial |\zeta-z|^2\wedge (\debar\partial |\zeta-z|^2)^{j-1}}{|\zeta-z|^{2j}} 
\end{equation*}
and we let $B_j$ be the component of $B$ of bidegree $(j,j-1)$.
Let $H=H_0+H_1+\cdots$ be a holomorphic form in $D_\zeta\times D_z$ with values in $\text{Hom}(E,E_0)$,
where $H_j$ has bidegree $(j,0)$ and values in $\text{Hom}(E_j,E_0)$. Let
$g=g_0+g_1+\cdots$ be a smooth form in $D''_\zeta\times D'_z$, where $g_j$ has bidegree $(j,j)$
and $D', D''\subset D$. The forms $H$ and $g$ will be specified in the next section.

If $\tau$ is a current in $D_\zeta\times D_z$ we let $(\tau)_N$ be the component of bidegree $(N,*)$ in $\zeta$ and 
$(0,*)$ in $z$. Let $\vartheta(\tau)$ be the current defined by
\begin{equation*}
(\tau)_N=\vartheta(\tau)\wedge d\zeta.
\end{equation*}
Notice that in view of \eqref{nyttnr},  
\begin{equation*}
(g\wedge HR)_N=\vartheta(g\wedge H)\mathcal{R};
\end{equation*}
here and for the rest of this section, $R=R(\zeta)$ and $\mathcal{R}=\mathcal{R}(\zeta)$.
Similarly, outside the diagonal $\Delta\subset D_\zeta\times D_z$,
\begin{equation*}
(B\wedge g\wedge HR)_N=\vartheta(B\wedge g\wedge H)\mathcal{R}.
\end{equation*}

Let $\varphi\in\VV_X^{p,*}$ and let $\mu\in\W_X^{n-p,*}$. We give a meaning to 
\begin{equation}\label{Ppreldef}
\vartheta(g\wedge H)\mathcal{R}\wedge \varphi(\zeta)\wedge i_*\mu(z)
\end{equation}
as follows. By Proposition~\ref{absolut}, $\mathcal{R}=a\wedge i_*\omega_0$ where $a$ is almost semi-meromorphic and generically smooth
on $Z$. Therefore, by Lemma~\ref{rutat}, $\mathcal{R}\wedge\varphi:=a\wedge i_*(\varphi\wedge\omega_0)$ is a well-defined current in
$\W_D^{Z,*}$. Since $\mathcal{R}\wedge \varphi(\zeta)\wedge i_*\mu(z)$ exists as a tensor product and $\vartheta(g\wedge H)$ is smooth,
\eqref{Ppreldef} is defined. Notice that it is annihilated by both $\J(\zeta)$ and $\J(z)$, i.e., it is 
$\hol_X$-linear both in $\varphi$ and $\mu$.
Moreover, by \cite[Corollary~4.7]{AWreg} it is in $\PM_{D''\times D'}$,
has support in $Z\times Z$ and the SEP with respect to $Z\times Z$.

Let $\pi^i\colon D_\zeta\times D_z\to D$, $i=1,2$, be the natural projections on the first and second factor, respectively. 
If $\tau$ is a current in $D\times D$ such that $\pi^i$ is proper on the support of $\tau$, then 
$\pi^i_{*}\tau$ is a current in $D$. Moreover, in view of \eqref{SEPprojformel}, if $\tau\in\PM_{D\times D}$ has support in $Z\times Z$
and the SEP with respect to $Z\times Z$, then
$\pi^i_{*}\tau\in \PM_D$ has support in $Z$ and the SEP with respect to $Z$.

\begin{definition}[The operators $P$ and $\check P$]\label{Pdefs}
If $g$ is smooth in $D\times D'$ (i.e., $D''=D$) and $\zeta\mapsto g(\zeta,z)$ has support in a fixed 
compact subset of $D$ for all $z\in D'$, we define 
$P\colon \VV^{p,*}(X)\to \VV^{p,*}(X\cap D')$ by
\begin{equation}\label{kork}
i_* P\varphi\wedge\mu = \pi^2_{*} \big(\vartheta(g\wedge H)\mathcal{R}\wedge \varphi(\zeta)\wedge i_*\mu(z)\big), \quad
\varphi\in\VV^{p,*}(X), \,\,\mu\in\Ba^{n-p}(X\cap D').
\end{equation} 

\smallskip

If $g$ is smooth in $D''\times D$ (i.e., $D'=D$) and $z\mapsto g(\zeta,z)$ has support in a fixed compact subset of $D$ for all $\zeta\in D''$,
we define $\check P\colon \W^{n-p,*}(X)\to \W^{n-p,*}(X\cap D'')$ by
\begin{equation}\label{korka}
i_* \check P\mu = \pi^1_{*} \big(\vartheta(g\wedge H)\mathcal{R}\wedge i_*\mu(z)\big), \quad
\mu\in\W^{n-p,*}(X).
\end{equation}
\end{definition}

\smallskip

If $\zeta\mapsto g(\zeta,z)$ does not have compact support then $P\varphi$  is still well-defined by \eqref{kork} if 
$\varphi$ has compact support in $X$. Similarly, if $z\mapsto g(\zeta,z)$ does not have compact support then $\check P\mu$ 
is well-defined by \eqref{korka} if $\mu$ has compact support. 

Notice that $i_* P\varphi$ is a smooth $(p,*)$-form in $D'$ since $\vartheta(g\wedge H)\mathcal{R}$ is smooth in $z$ and takes values in $E_0$; 
if $g$ is holomorphic in $z$, then $i_* P\varphi$ is holomorphic. Moreover,
since $\mathcal{R}=\mathcal{R}(\zeta)$, it follows that $i_* \check P\mu=\psi\wedge\mathcal{R}$ for some smooth form $\psi$ in $D''$.

\smallskip

To define the operators $K$ and $\check K$ notice first that, in a similar way as for $P$ and $\check P$,
we can give a meaning to
\begin{equation}\label{Kpreldef}
\vartheta(B\wedge g\wedge H)\mathcal{R}\wedge \varphi(\zeta)\wedge i_*\mu(z)
\end{equation} 
outside the diagonal $\Delta\subset D\times D$ since $B$ is smooth there.

\begin{lemma}\label{yster}
The current \eqref{Kpreldef} has a unique extension to a current in $\PM_{D\times D}$ with support in $Z\times Z$
and the SEP with respect to $Z\times Z$. The extension is 
annihilated by both $\J(\zeta)$ and $\J(z)$.
\end{lemma}

\begin{proof}
The uniqueness is clear by the SEP since \eqref{Kpreldef} a priori is defined in $D\times D\setminus\Delta$ and has support in
$Z\times Z\setminus \Delta$.

Recall that $\mathcal{R}\wedge\varphi(\zeta)\wedge i_*\mu(z)\in\PM_{D\times D}$ has support in $Z\times Z$ and the SEP
with respect to $Z\times Z$. Since $B$ is almost semi-meromorphic 
in $D\times D$, also $\vartheta(B\wedge g\wedge H)$ has these properties. Hence, cf.\ \eqref{ASMprod},
$\vartheta(B\wedge g\wedge H)\mathcal{R}\wedge\varphi(\zeta)\wedge i_*\mu(z)$ is in $\PM_{D\times D}$ with
support in $Z\times Z$ and the SEP with respect to $Z\times Z$.

Clearly $\J(\zeta)$ and $\J(z)$ annihilate \eqref{Kpreldef} outside $\Delta$. Since the extension 
has the SEP with respect to $Z\times Z$ it is annihilated by $\J(\zeta)$ and $\J(z)$.
\end{proof}

We will use the notation \eqref{Kpreldef} to denote the extension as well.
In view of the lemma it depends $\hol_X$-linearly on both $\varphi$ and $\mu$.

\begin{definition}[The operators $K$ and $\check K$]\label{Kdefs}
If $g$ is smooth in $D\times D'$ (i.e., $D''=D$) and $\zeta\mapsto g(\zeta,z)$ has support in a fixed compact 
subset of $D$ for all $z\in D'$, we define 
$K\colon \VV^{p,*}(X)\to \VV^{p,*}(X\cap D')$ by
\begin{equation*}
i_* K\varphi\wedge\mu = \pi^2_{*} \big(\vartheta(B\wedge g\wedge H)\mathcal{R}\wedge \varphi(\zeta)\wedge i_*\mu(z)\big), \quad
\varphi\in\VV^{p,*}(X), \,\,\mu\in\Ba^{n-p}(X\cap D').
\end{equation*} 

\smallskip

If $g$ is smooth in $D''\times D$ (i.e., $D'=D$) and $z\mapsto g(\zeta,z)$ has support in a fixed compact subset of $D$ for all $\zeta\in D''$,
we define $\check K\colon \W^{n-p,*}(X)\to \W^{n-p,*}(X\cap D'')$ by
\begin{equation*}
i_* \check K\mu = \pi^1_{*} \big(\vartheta(B\wedge g\wedge H)\mathcal{R}\wedge i_*\mu(z)\big), \quad
\mu\in\W^{n-p,*}(X).
\end{equation*}
\end{definition}

\smallskip

As with the operators $P$ and $\check P$, if $\varphi$ and $\mu$ have compact support in $X$, then $K\varphi$ and $\check K\mu$ are defined 
also when $\zeta\mapsto g(\zeta,z)$ and $z\mapsto g(\zeta,z)$, respectively, do not have compact support.

\begin{theorem}\label{knall}
(i) If $\varphi\in\V_X^{p,*}$ is in $\E_X^{p,*}$ in a neighborhood of a point $x\in \Xpreg$, then $K\varphi$ is in $\E_X^{p,*}$ in a neighborhood
of $x$.

(ii) Assume that $\mu\in\W_X^{n-p,*}$ is such that, in a neighborhood of $x\in\Xpreg$, $i_*\mu=\sum_\ell\mu_\ell\wedge i_*\omega_\ell$, 
where $\mu_\ell \in\E_D^{0,*}$ and $\omega_\ell\in \Ba_X^{n-p}$. Then $i_*\check K \mu$ is of  the same form in a neighborhood of $x$.
\end{theorem}

Recall that, by Proposition~\ref{hard}, in a neighborhood of $x\in\Xpreg$, any $\phi\in\V_X^{p,*}$
is represented by $\sum_k \phi_k\wedge b_k$ for some $\phi_k\in\W_Z^{0,*}$. That $\phi\in\V_X^{p,*}$
is smooth means, cf.\ Lemma~\ref{smirnoff}, that $\phi_k\in\E_Z^{0,*}$. 
In view of this it is natural to call a $\mu\in\W_X^{n-p,*}$ with the property in (ii) smooth. Analogously to part (i), part (ii) 
of the theorem thus says 
that $\check K$ preserves the smooth elements of $\W_X^{n-p,*}$. 

\begin{proof}
Notice that if $\varphi=\varphi(\zeta)\equiv 0$ in a neighborhood of $x$,
then $K\varphi$
is smooth in a neighborhood of $x$ since in that case $\vartheta(B\wedge g\wedge H)\mathcal{R}\wedge\varphi$ is smooth for $z$ in 
a neighborhood of $x$. To prove the first part of the theorem we may thus assume that $\varphi$ has support in a small neighborhood of $x$.

Let $(z,w)$ and $(\zeta,\tau)$ be two sets of the same local coordinates in $D$ centered at $x$ such that $Z=\{w=0\}=\{\tau=0\}$ 
in a neighborhood of $x$;
these coordinates need not have any relation to our previous local coordinates which were used to define $B$.
Suppose that $\varphi$ has support where the coordinates $(z,w)$ are defined. 
Let $\chi^\epsilon:=\chi(|\zeta-z|^2/\epsilon)$ and let, for any $\mu\in\W_X^{n-p,*}$,
\begin{equation}\label{Tee}
T:=\vartheta(B\wedge g\wedge H)\mathcal{R}\wedge\varphi(\zeta,\tau)\wedge i_*\mu(z,w).
\end{equation} 
Then, in view of  \eqref{tabasco},
\begin{equation*}
\lim_{\epsilon\to 0} \chi^\epsilon T= 
\mathbf{1}_{D\times D\setminus \{\zeta=z\}} T.
\end{equation*}
By Lemma~\ref{yster}, $T$ has the SEP with respect to $Z\times Z$ and so, since $\{\zeta=z\}\cap Z\times Z$ is a proper subset of $Z\times Z$,
$\mathbf{1}_{\{\zeta=z\}}T=0$. Hence, $\mathbf{1}_{D\times D\setminus \{\zeta=z\}} T=T$ and thus $\chi^\epsilon T\to T$.
Define $K^\epsilon\varphi$ by
\begin{equation}\label{Te}
K^\epsilon\varphi:=\pi^2_{*} \big(\chi^\epsilon\vartheta(B\wedge g\wedge H)\mathcal{R}\wedge\varphi(\zeta,\tau)\big).
\end{equation}
Then $K^\epsilon\varphi$ is smooth since $\chi^\epsilon\vartheta(B\wedge g\wedge H)\mathcal{R}\wedge\varphi(\zeta,\tau)$ is smooth in $(z,w)$ 
and it follows that
\begin{equation}\label{oxe}
K^\epsilon \varphi\wedge i_*\mu = \pi^2_{*}  (\chi^\epsilon T) \to \pi^2_{*}T = i_* K\varphi\wedge\mu
\end{equation}
as currents in $D'$.

By Lemma~\ref{smirnoff} there are $\phi_k^\epsilon\in\E_Z^{0,*}$ such that 
\begin{equation*}
K^\epsilon\varphi=\sum_k\phi_k^\epsilon\wedge b_k + \Kernel_p i^*
\end{equation*}
and by the proof of that lemma $\phi_k^\epsilon$ are obtained by applying linear combinations of 
$\partial^{|\alpha|}/\partial w^\alpha$ to (the coefficients)
of $K^\epsilon\varphi$ and evaluate at $w=0$. We claim that there are $\phi_k\in\E_Z^{0,*}$ such that 
$\phi_k^\epsilon\to\phi_k$ as currents on $Z$. 

Given the claim we can conclude the proof of the first part of the theorem. Let $\mu\in\Ba_X^{n-p}$. Then $b_k\wedge i_*\mu\in\CH_D^{Z}$ and so,
in view of \eqref{klippning} and the representation \eqref{grymta} of $b_k\wedge i_*\mu$, there are $a_{k,\alpha}(z)\in\Om_Z^n$ such that 
$b_k\wedge i_*\mu=\sum_\alpha a_{k,\alpha}(z)\wedge \debar(dw/w^{\alpha+\mathbf{1}})$. Hence,
\begin{eqnarray*}
K^\epsilon\varphi\wedge i_*\mu &=& \sum_k\phi_k^\epsilon\wedge b_k\wedge i_*\mu
=\sum_{k,\alpha} \phi_k^\epsilon(z)\wedge a_{k,\alpha}(z)\wedge \debar\frac{dw}{w^{\alpha+\mathbf{1}}} \\
&\to & 
\sum_{k,\alpha} \phi_k(z)\wedge a_{k,\alpha}(z)\wedge \debar\frac{dw}{w^{\alpha+\mathbf{1}}}
= \sum_k \phi_k\wedge b_k \wedge i_*\mu
\end{eqnarray*}
as currents in $D'$. In view of \eqref{oxe}, thus
\begin{equation*}
i_* K\varphi\wedge\mu = \sum_k \phi_k\wedge b_k \wedge i_*\mu,
\end{equation*} 
which means that $K\varphi\in\V_X^{p,*}$ is smooth.

To show the claim, notice that since $\mathcal{R}=\mathcal{R}_\kappa + \mathcal{R}_{\kappa+1}+\cdots$
we can replace $H$ in \eqref{Te} by $H_\kappa + H_{\kappa+1}+\cdots$. Hence, only $B_j$ with $j\leq N-\kappa=n$
contribute in \eqref{Te}. In view of Proposition~\ref{absolut}, $\mathcal{R}=a\cdot i_*\omega_0$, where $a$ is smooth 
on $\Xpreg$ and $\omega_0\in\Ba_X^{n-p}$. Since $i_*\omega\in\Hom_{\hol_D}(\Om_D^p,\CH_D^Z)$, cf.\ \eqref{grymta},
it follows that $K^\epsilon\varphi$ is a sum of terms of the form
\begin{equation}\label{snaps}
\pi^2_{*}\big( \chi^\epsilon B_j\wedge\phi(\zeta,\tau,z,w)\wedge \debar\frac{d\tau}{\tau^{\beta+\mathbf{1}}}\big),
\end{equation}
where $j\leq n$ and $\phi$ is smooth with support in a neighborhood of $(\zeta,\tau)=x$.
It is proved in \cite[Proposition~10.5]{AL}, cf.\ in particular \cite[Equation~(10.5)]{AL}, 
that after applying $\partial^{|\alpha|}/\partial w^\alpha$
to a term \eqref{snaps} and evaluating at $w=0$ the limit as $\epsilon\to 0$ is smooth in $z$. The claim thus follows.

\smallskip

The proof of part (ii) of the theorem is similar. First notice that if $\mu\equiv 0$ in a neighborhood of 
$x$, then $i_*\check K\mu$ equals $\mathcal{R}$ times a smooth form in a neighborhood of $x$. 
Since $\mathcal{R}=a\cdot i_*\omega_0$, where $a$ is smooth on $\Xpreg$, the second part follows in this case. 
We can thus assume that $\mu$ has support in a small neighborhood of $x$.

Let $T$ be given by \eqref{Tee} with $\varphi=1$. As above it follows that $\chi^\epsilon T\to T$.
Set 
\begin{equation*}
u^\epsilon:=\pi^1_{*} \big(\chi^\epsilon \vartheta(B\wedge g\wedge H)\wedge i_*\mu(z)\big). 
\end{equation*}
Then $u^\epsilon$ is smooth and it follows that
\begin{equation}\label{regn}
u^\epsilon\wedge\mathcal{R}\to i_*\check K\mu
\end{equation}
as currents in $D''$. As in the proof of Lemma~\ref{smirnoff} we have
\begin{equation*}
u^\epsilon(\zeta,\tau)=\sum_{|\beta|<M} \frac{\partial^{|\beta|}u^\epsilon}{\partial \tau^\beta}(\zeta,0)\frac{\tau^\beta}{\beta !}
+\mathcal{O}(|\tau|^M,\bar{\tau},d\bar{\tau}),
\end{equation*}
where $\mathcal{O}(|\tau|^M,\bar{\tau},d\bar{\tau})$ is a sum of terms which are either $\mathcal{O}(|\tau|^M)$ or divisible
by some $\bar{\tau}_j$ or $d\bar{\tau}_j$. Since $\mathcal{O}(|\tau|^M,\bar{\tau},d\bar{\tau})\wedge\mathcal{R}=0$,
if there are $u_\beta(\zeta)\in\E_X^{0,*}$ such that $\partial^{|\beta|}u^\epsilon(\zeta,0)/\partial \tau^\beta\to u_\beta(\zeta)$
as current on $Z$, it follows as above that 
\begin{equation*}
u^\epsilon\wedge\mathcal{R}\to \sum_{|\beta|<M} u_\beta(\zeta)\tau^\beta\wedge i_*\omega_0/\beta !
\end{equation*}
as currents in $D''$. Thus, by \eqref{regn}, $i_*\check K\mu$ has the desired form. To see that there are such $u_\beta$, 
notice that if $i_*\mu=\sum_\ell \mu_\ell\wedge i_*\omega_\ell$ then $u^\epsilon$ is a sum of terms 
\begin{equation*}
\pi^1_{*}\big(\chi^\epsilon B_j\wedge\phi(\zeta,\tau,z,w)\wedge \debar\frac{dw}{w^{\alpha+\mathbf{1}}}\big),
\end{equation*}
where $j\leq n$ and $\phi$ is smooth, cf.\ \eqref{snaps} and the preceding argument. 
The existence of such $u_\beta$ thus follows as above. 
\end{proof}

The following lemma will be used in the next section. It can be proved along the same lines as \cite[Lemma~9.5]{AL}.
We remark that the latter lemma 
is formulated and proved in terms of a $\lambda$-regularization of $R$. 
However, in view of \cite[Lemma~6]{LS}, $\lambda$-regularization can be replaced by $\epsilon$-regularizations.

\begin{lemma}\label{vittling}
Let $R^\epsilon:=\debar \chi(|F|^2/\epsilon)\wedge u$, cf.\ \eqref{tangent}, and let $\mathcal{R}^\epsilon:=R^\epsilon \otimes d\zeta$.
Then
\begin{equation*}
\lim_{\epsilon\to 0} \mathcal{R}(z)\wedge\vartheta(B\wedge g\wedge H)\mathcal{R}^\epsilon =
\mathcal{R}(z)\wedge\vartheta(B\wedge g\wedge H)\mathcal{R},
\end{equation*}
where the right-hand side is the product of the almost semi-meromorphic current $\vartheta(B\wedge g\wedge H)$ 
by the tensor product $\mathcal{R}(z)\wedge\mathcal{R}$,
cf.\ Lemma~\ref{yster}. 
\end{lemma}


\section{Koppelman formulas and the sheaves $\A_X^{p,*}$ and $\Bsheaf_X^{n-p,*}$}\label{AoBsektion}
We assume now that $i\colon X\to D\subset\C^N$ is a local embedding into a pseudoconvex open set.
Let $z$ and $\zeta$ be two sets of the same local coordinates in $D$ and let $B$ be the corresponding
Bochner--Martinelli form. We choose $g$ and $H$ in the definition of the integral operators 
of Section~\ref{intop} to be a \emph{weight}, in the sense of \cite[Section~2]{AintrepII}, and a \emph{Hefer morphism},
in the sense of \cite[Section~5]{AW1} and \cite[Proposition~5.3]{AintrepII}, respectively.

\begin{example}[Example~2 in \cite{AintrepII}]\label{viktigt}
Let $D'\Subset D$ and assume that $\overline{D}'$ is holomorphically convex.
Let $\chi$ be a cutoff function in $D$ such that $\chi=1$ in a neighborhood of $\overline{D}'$.
One can find a smooth $(1,0)$-form $s(\zeta,z)=\sum_j s_j(\zeta,z)d(\zeta_j-z_j)$, defined for $\zeta$ in a neighborhood of
$\text{supp}\,\debar\chi$ and $z$ in a neighborhood of $\overline{D}'$, such that $2\pi i\sum_j (\zeta_j-z_j)s_j(\zeta,z)=1$
and $z\mapsto s(\zeta,z)$ is holomorphic. Then
\begin{equation*}
g=\chi(\zeta) - \debar\chi(\zeta)\wedge\sum_{k=1}^{N} s(\zeta,z)\wedge (\debar s(\zeta,z))^{k-1}
\end{equation*}
is a weight with compact support in $D_\zeta$, depends holomorphically on $z$ in a neighborhood of $\overline{D}'$, and
contains no $d\bar{z}_j$.

If $D'$ is the unit ball we can take $s(\zeta,z)=\sum_j(\bar{\zeta}_j-\bar{z}_j)d(\zeta_j-z_j)/2\pi i(|\zeta|^2-z\cdot \bar{\zeta})$ 
\end{example}

Let $\omega$ be an $(n-p)$-structure form on $X$ and recall, see \eqref{strukturform}, that $i_*\omega=\mathcal{R}$
for $\mathcal{R}$ associated to a Hermitian resolution \eqref{karvkplx}.
Then by Proposition~\ref{absolut}, $i_*\omega=a\cdot i_*\omega_0$ for some tuple 
$\omega_0$ of elements in $\Ba_X^{n-p}$ and a matrix of almost semi-meromorphic currents $a$ which is smooth on $\Xpreg$.
In view of Lemma~\ref{rutat} it follows that if $\varphi\in\V_X^{p,*}$, then $\varphi\wedge\omega$ is well-defined in $\W_X^{n,*}$.

\begin{definition}\label{Dom}
If $\varphi\in\V_X^{p,*}$ we say that $\varphi\in\text{Dom}\,\debar_X$ if
$\debar(\varphi\wedge\omega)\in\W_X^{n,*}$ for any $(n-p)$-structure form $\omega$ on $X$.
\end{definition}

Let us point out a few consequences. 
We can define $\debar\colon \text{Dom}\,\debar_X \to \V_X^{p,*}$ as follows.
Let  $\mu\in\Ba_X^{n-p}$. In view of \eqref{barlet-bjork}, since the map \eqref{bibblan} is an isomorphism,
there is a current $\mathcal{R}$ and a holomorphic $E^*$-valued function $\xi$ such that $i_*\mu=\xi\cdot\mathcal{R}$.
Thus, by \eqref{strukturform} there is an $(n-p)$-structure form $\omega$ such that $\mu=i^*\xi\cdot\omega$.
If $\varphi\in\text{Dom}\,\debar_X$ 
it follows that $\debar(\varphi\wedge\mu)\in\W_X^{n,*}$.
Hence, for $\varphi\in\text{Dom}\,\debar_X$ we can define $\debar\varphi\in \V_X^{p,*}$ by
\begin{equation*}
\debar\varphi\wedge\mu := \debar (\varphi\wedge\mu),\quad  \mu\in\Ba_X^{n-p}.
\end{equation*}

Since $\debar\varphi\in\V_X^{p,*}$ if $\varphi\in\text{Dom}\,\debar_X$ it follows as in the paragraph preceding 
Definition~\ref{Dom} that $\debar\varphi\wedge\omega$ is well-defined in $\W_X^{n,*}$ for any $(n-p)$-structure form $\omega$.
Moreover, if as above $i_*\omega=\mathcal{R}=a\cdot i_*\omega_0$, where $\mathcal{R}$ is associated to the Hermitian resolution \eqref{karvkplx},
then 
\begin{equation}\label{telefjomp}
\debar\varphi\wedge\omega=-\nabla_f (\varphi\wedge\omega),
\end{equation}
where $\nabla_f=f-\debar$.
In fact, by Lemma~\ref{nissen}, $f a\cdot i_*\omega_0=\debar (a\cdot i_*\omega_0)$ and so, since $a$ is smooth on $\Xpreg$,
in view of Lemma~\ref{rutat}, we get
\begin{eqnarray*}
-\nabla_f (\varphi\wedge\omega) &=& \debar(\varphi\wedge i^*a\cdot\omega_0) - f(\varphi\wedge i^*a\cdot\omega_0) \\
&=&
\pm i^*a\cdot \debar(\varphi\wedge\omega_0) \pm \varphi\wedge\debar (i^*a\cdot\omega_0) \mp \varphi\wedge fi^*a\cdot\omega_0 \\
&=&
\pm i^*a\cdot \debar(\varphi\wedge\omega_0) = \debar\varphi\wedge i^*a\cdot\omega_0 = \debar\varphi \wedge\omega
\end{eqnarray*}
on $\Xpreg$. Since both sides of \eqref{telefjomp} have the SEP, \eqref{telefjomp} holds everywhere.

We also notice that  
\begin{equation}\label{fredag}
\E_X^{p,*}\subset\text{Dom}\,\debar_X.
\end{equation}
This follows since, as above, any $(n-p)$-structure form $\omega$ satisfies $\debar\omega=f\omega$ for an appropriate $f$
and hence, if $\varphi\in\E_X^{p,*}$, $\debar(\varphi\wedge\omega)=\debar\varphi\wedge\omega\pm \varphi\wedge f\omega\in\W_X^{n,*}$.


\begin{proposition}\label{kolja}
Let $D'\Subset D$ be a relatively compact open subset and set $X'=X\cap D'$.
There are integral operators
\begin{equation*}
K\colon \E^{p,*+1}(X)\to\V^{p,*}(X')\cap\text{Dom}\,\debar_X, \quad
P\colon \E^{p,*}(X)\to\E^{p,*}(X')
\end{equation*}
such that for any $\varphi\in\E^{p,*+1}(X)$,
\begin{equation}\label{ortodox}
\varphi=\debar K\varphi + K\debar\varphi + P\varphi.
\end{equation}

If $\varphi\in\E^{p,*+1}(X)$ has compact support in $X$ one can choose $K$ and $P$ such that, additionally,  $K\varphi$ and $P\varphi$
have compact support in $X$.
\end{proposition}

\begin{proposition}\label{haddock}
Let $D'\Subset D$ be a relatively compact open subset and set $X'=X\cap D'$.
There are integral operators
\begin{equation*}
\check K\colon \W^{n-p,*+1}(X)\to\W^{n-p,*}(X'), \quad
\check P\colon \W^{n-p,*}(X)\to\W^{n-p,*}(X')
\end{equation*}
such that if $i_*\mu = \sum_\ell \mu_\ell\wedge i_*\omega_\ell$ for some $\mu_\ell\in\E_D^{0,*}$ and $\omega_\ell\in\Ba_X^{n-p}$, then 
\begin{equation}\label{oortodox}
\mu=\debar \check K\mu + \check K\debar\mu + \check P\mu.
\end{equation}

If $\mu$ in addition has compact support in $X$ one can choose $\check K$ and $\check P$ such that $\check K\varphi$ and $\check P\varphi$
have compact support in $X$.
\end{proposition}

\begin{proof}[Proof of Propositions~\ref{kolja} and \ref{haddock}]
Let $\mathcal{R}^\epsilon$ be as in Lemma~\ref{vittling}. In the same way as in \cite[Section~5]{SK}, cf.\ also 
\cite[Section~5]{AS} and \cite[Eq.~(9.16)]{AL}, one obtains
\begin{equation}\label{basta}
\nabla_{f(z)} \big(\mathcal{R}(z)\wedge\vartheta(B\wedge g\wedge H)\mathcal{R}^\epsilon\big) =
\mathcal{R}\wedge[\Delta] - \mathcal{R}(z)\wedge\vartheta(g\wedge H)\mathcal{R}^\epsilon,
\end{equation}
where $\nabla_{f(z)}=f(z)-\debar$. Notice that, for $\epsilon>0$, all current products are well-defined as tensor products.
Letting $\epsilon\to 0$ we get by Lemma~\ref{vittling},
\begin{equation}\label{makrill}
\nabla_{f(z)} \big(\mathcal{R}(z)\wedge\vartheta(B\wedge g\wedge H)\mathcal{R}\big) =
\mathcal{R}\wedge[\Delta] - \mathcal{R}(z)\wedge\vartheta(g\wedge H)\mathcal{R}.
\end{equation}

To show the first statement of Proposition~\ref{kolja}, choose $g$ such that $\zeta\mapsto g(\zeta,z)$ has support in a 
fixed compact subset, containing $D'$, of $D$ for all $z\in D'$. Multiplying \eqref{makrill} by a $\tilde\varphi(\zeta)\in\E^{p,*+1}(D)$ 
such that $i^*\tilde\varphi=\varphi$ and 
applying $\pi^2_{*}$ we get
\begin{equation*}
\nabla_f (\mathcal{R}\wedge i_* K\varphi) + \mathcal{R}\wedge i_*K(\debar \varphi) =
\mathcal{R}\wedge \tilde\varphi - \mathcal{R}\wedge i_*P\varphi,
\end{equation*}
i.e.,
\begin{equation}\label{telefontratt}
\nabla_f (\omega\wedge  K\varphi) + \omega\wedge K(\debar \varphi) =
\omega\wedge \varphi - \omega\wedge P\varphi.
\end{equation}
In view of Definitions~\ref{Pdefs} and \ref{Kdefs} all terms except $\nabla_f (\omega\wedge  K\varphi)$
are in $\W_X^{n,*}$ and consequently $\nabla_f (\omega\wedge  K\varphi)$ is too. 
Hence, since $f(\omega\wedge K\varphi)\in \W_X^{n,*}$ also $\debar(\omega\wedge  K\varphi)\in \W_X^{n,*}$,
and so $K\varphi\in \text{Dom}\,\debar_X$. Thus, by \eqref{telefjomp}, we can replace $\nabla_f (\omega\wedge  K\varphi)$
in \eqref{telefontratt} by $-\omega\wedge\debar K\varphi$. Multiplying the resulting equality by holomorphic $E^*$-valued $\xi$
such that $f^*\xi=0$ we get, since the map \eqref{bibblan} is an isomorphism, 
\begin{equation}\label{snabeltand}
\mu\wedge\varphi=\mu\wedge\debar K\varphi + \mu\wedge K\debar\varphi + \mu\wedge P\varphi, \quad \forall\mu\in\Ba_X^{n-p},
\end{equation} 
which is what \eqref{ortodox} means.

If $\varphi$ has compact support we can take a weight $g$ such that $z\mapsto g(\zeta,z)$ has compact support.
The preceding argument goes through unchanged and it is clear that $K\varphi$ and $P\varphi$ have compact support.

\smallskip

To show Proposition~\ref{haddock}, let $\xi_\ell$ be holomorphic $f^*$-closed sections of $E^*$ such that
$i_*\omega_\ell=\xi_\ell\cdot \mathcal{R}$, so that $i_*\mu=\sum_\ell \mu_\ell\wedge\xi_\ell\cdot\mathcal{R}$.
Since $\nabla_f\mathcal{R}=0$ and $\debar\xi_\ell=0$ a simple computations gives
\begin{eqnarray*}
\mu_\ell\wedge\xi_\ell\cdot\nabla_{f(z)} \big(\mathcal{R}(z)\wedge\vartheta(B\wedge g\wedge H)\mathcal{R}^\epsilon\big)
&=&
\debar\big(\mu_\ell\wedge\xi_\ell\cdot \mathcal{R}(z)\wedge\vartheta(B\wedge g\wedge H)\mathcal{R}^\epsilon\big)\\
& & + \debar\mu_\ell\wedge\xi_\ell\cdot \mathcal{R}(z)\wedge\vartheta(B\wedge g\wedge H)\mathcal{R}^\epsilon.
\end{eqnarray*}
Hence, in view of Lemma~\ref{vittling}, multiplying \eqref{basta} by $\sum_\ell \mu_\ell\wedge\xi_\ell$ and letting
$\epsilon\to 0$ we obtain
\begin{equation*}
\debar\big(i_*\mu(z)\wedge\vartheta(B\wedge g\wedge H)\mathcal{R}\big) + 
\debar i_*\mu(z)\wedge\vartheta(B\wedge g\wedge H)\mathcal{R} =
i_*\mu\wedge[\Delta] - i_*\mu\wedge\vartheta(g\wedge H)\mathcal{R}.
\end{equation*}

If $g$ is chosen so that $z\mapsto g(\zeta,z)$ has support in a fixed compact for all $\zeta\in D'$, then
\eqref{oortodox} follows by applying $\pi^1_{*}$. If $\mu$ has compact support we instead choose $g$ 
such that $\zeta\mapsto g(\zeta,z)$ has compact support and apply $\pi^1_{*}$. 
\end{proof}


\begin{definition}\label{Adef}
If $\mathcal{U}\subset X$ is open and $\varphi\in\V^{p,q}(\mathcal{U})$ we say that $\varphi$ is a section of 
$\A_X^{p,q}$ over $\mathcal{U}$, $\varphi\in\A^{p,q}(\mathcal{U})$, if for every $x\in\mathcal{U}$ the germ $\varphi_x$
can be written as a finite sum of terms
\begin{equation}\label{skurk}
\xi_\nu\wedge K_\nu(\cdots \xi_2\wedge K_2 (\xi_1\wedge K_1(\xi_0))\cdots),
\end{equation}
where $\nu\geq 0$, $\xi_0\in\E_X^{p,*}$, $\xi_j\in\E_X^{0,*}$ for $j\geq 1$, $K_j$ are integral operators as defined in Section~\ref{intop},
and $\xi_j$ has compact support in the set where $z\mapsto K_j(\zeta,z)$ is defined.  
\end{definition}

\begin{definition}\label{Bdef}
If $\mathcal{U}\subset X$ is open and $\mu\in\W^{n-p,q}(\mathcal{U})$ we say that $\mu$ is a section of 
$\Bsheaf_X^{n-p,q}$ over $\mathcal{U}$, $\mu\in\Bsheaf^{n-p,q}(\mathcal{U})$, if for every $x\in\mathcal{U}$ the germ $\mu_x$
can be written as a finite sum of terms
\begin{equation}\label{tjyv}
\xi_\nu\wedge \check K_\nu(\cdots \xi_2\wedge \check K_2 (\xi_1\wedge \check K_1(\xi_0\wedge\omega))\cdots),
\end{equation}
where $\nu\geq 0$, $\omega$ is an $(n-p)$-structure form, $\xi_j\in\E_X^{0,*}$, $\check K_j$ are integral operators as defined in
Section~\ref{intop}, $\xi_j$ has compact support in the set where $\zeta\mapsto \check K_j(\zeta,z)$ is defined, and
$\xi_0$ takes values in $E^*$.
\end{definition}

\begin{proposition}\label{hemma}
The sheaf $\A_X^{p,*}$ has the following properties.
\begin{itemize}
\item[(a1)] It is a module over $\E_X^{0,*}$,
\item[(a2)] if $K$ is an integral operator as defined in Section~\ref{intop} then
$K\colon\A_X^{p,*+1}\to\A_X^{p,*}$,
\item[(a3)] $\A_X^{p,*}\subset \text{Dom}\,\debar_X$,
\item[(a4)] $\debar\colon\A_X^{p,*}\to\A_X^{p,*+1}$,
\item[(a5)] $\A_X^{p,*}=\E_X^{p,*}$ on $\Xpreg$,
\item[(a6)] \eqref{ortodox} holds for $\varphi\in\A_X^{p,*}$.
\end{itemize}

\noindent The sheaf $\Bsheaf_X^{n-p,*}$ has the following properties.
\begin{itemize}
\item[(b1)] It is a module over $\E_X^{0,*}$,
\item[(b2)] if $\check K$ is an integral operator as defined in Section~\ref{intop} then
$\check K\colon\Bsheaf_X^{n-p,*+1}\to\Bsheaf_X^{n-p,*}$,
\item[(b3)] $\debar\colon\Bsheaf_X^{n-p,*}\to\Bsheaf_X^{n-p,*+1}$,
\item[(b4)] if $\mu\in\Bsheaf_X^{n-p,*}$ then on $\Xpreg$, $\mu=\sum_\ell\mu_\ell\wedge \omega_\ell$ for some
$\mu_\ell\in\E_X^{0,*}$ and $\omega_\ell\in\Ba_X^{n-p}$,
\item[(b5)] \eqref{oortodox} holds for $\mu\in\Bsheaf_X^{n-p,*}$.
\end{itemize}
\end{proposition}

To prove this proposition we need the following two lemmas.
The first one extends Propositions~\ref{kolja} and \ref{haddock}.

\begin{lemma}\label{ultraortodox}
Let $\varphi\in\V^{p,*}(X)$. Assume that $\varphi, K\varphi\in\text{Dom}\,\debar_X$ and that $\varphi\in\E_X^{p,*}$ on $\Xpreg$. 
Then \eqref{ortodox} holds on $X'$.
If in addition $\varphi$ has compact support, then $K$ and $P$ can be chosen such that $K\varphi$ and $P\varphi$ have compact support.

Let $\mu\in\W^{n-p,*}(X)$. Assume that $\debar\mu, \debar\check K \mu\in\W_X^{n-p,*}$ and
that $i_*\mu=\sum_\ell \mu_\ell\wedge\omega_\ell$ on $\Xpreg$ for some $\mu_\ell\in\E_D^{0,*}$ and $\omega_\ell\in\Ba_X^{n-p,*}$.
Then \eqref{oortodox} holds on $X'$. If $\mu$ in addition has compact support, then $\check K$ and $\check P$ can be chosen such that
$\check K\mu$ and $\check P\mu$ have compact support.
\end{lemma}

\begin{proof}
Let $h$ be a holomorphic tuple vanishing precisely on $X_{p\text{-sing}}$ and set $\chi^\epsilon=\chi(|h|^2/\epsilon)$. 
Then Proposition~\ref{kolja} applies to $\chi^\epsilon\varphi$ and hence
\begin{equation*}
\chi^\epsilon\varphi = \debar K(\chi^\epsilon\varphi) + K(\chi^\epsilon\debar\varphi) + K(\debar\chi^\epsilon\wedge\varphi)
+ P(\chi^\epsilon\varphi);
\end{equation*}
recall that this means that \eqref{snabeltand}, with $\varphi$ replaced by $\chi^\epsilon\varphi$, holds.
Since $\varphi\in\V_X^{p,*}$ it follows that $\chi^\epsilon\varphi\to\varphi$, i.e., $\chi^\epsilon\varphi\wedge\mu\to\varphi\wedge\mu$
for all $\mu\in\Ba_X^{n-p}$. By Lemma~\ref{yster} the current \eqref{Kpreldef} is in 
has the SEP with respect to $Z\times Z$ and therefore $K(\chi^\epsilon\varphi)\to K\varphi$. Similarly,
$P(\chi^\epsilon\varphi)\to P\varphi$. Moreover, $\debar\varphi\in\V_X^{p,*}$ since $\varphi\in\text{Dom}\,\debar_X$ and so
$K(\chi^\epsilon\debar\varphi)\to K(\debar\varphi)$. We claim that $\lim_{\epsilon\to 0} K(\debar\chi^\epsilon\wedge\varphi)=0$
on $\Xpreg$. Given the claim it follows that \eqref{ortodox} holds on $\Xpreg$. Since $K\varphi\in\text{Dom}\,\debar_X$ by assumption
it follows by the SEP that \eqref{ortodox} holds. 

To show that claim we may assume that $z$ is in a fixed compact subset of $\Xpreg$.
Then $B\wedge\debar\chi^\epsilon(\zeta)$ is smooth if $\epsilon$ is small enough. 
It thus suffices to show that 
\begin{equation*}
\mathcal{R}(\zeta)\wedge\debar\chi^{\epsilon}(\zeta)\wedge\varphi(\zeta)\wedge i_*\mu(z)\to 0,\quad \forall \mu\in\Ba_X^{n-p}.
\end{equation*}
Since this is a tensor product it suffices to see that $\omega\wedge\debar\chi^{\epsilon}\wedge\varphi\to 0$. 
However, $\varphi\wedge\omega\in\W_X^{n,*}$ and, since $\varphi\in\text{Dom}\,\debar_X$, $\nabla_f(\varphi\wedge\omega)\in\W_X^{n,*}$.
In view of \eqref{telefjomp} thus
\begin{eqnarray*}
\debar\chi^{\epsilon}\wedge\varphi\wedge\omega &=&
\debar(\chi^\epsilon\varphi)\wedge\omega - \chi^\epsilon\debar\varphi\wedge\omega =
-\nabla_f(\chi^\epsilon\varphi\wedge\omega) + \chi^\epsilon\nabla_f(\varphi\wedge\omega) \\
& \to &
-\nabla_f(\varphi\wedge\omega) +\nabla_f(\varphi\wedge\omega) =0.
\end{eqnarray*}

The proof of the second part of the lemma is similar: By Proposition~\ref{haddock},
\begin{equation*}
\chi^\epsilon \mu= \debar \check K(\chi^\epsilon\mu) + \check K (\chi^\epsilon\debar\mu) + \check K(\debar\chi^\epsilon\wedge\mu)
+\check P(\chi^\epsilon\mu).
\end{equation*}
Since $\mu\in\W_X^{n-p,*}$ we have $\chi^\epsilon \mu\to\mu$. By Lemma~\ref{yster} the current \eqref{Kpreldef} has the SEP 
with respect to $Z\times Z$ and therefore $\check K(\chi^\epsilon\mu)\to \check K\mu$. Similarly,
$\check P(\chi^\epsilon\mu)\to\check P\mu$ and, since $\debar\mu\in\W_X^{n-p,*}$, $\check K (\chi^\epsilon\debar\mu)\to\check K\debar\mu$.
Hence, \eqref{oortodox} holds modulo $\tau:=\lim_{\epsilon\to 0}\check K(\debar\chi^\epsilon\wedge\mu)$.
Since $\debar\check K\mu\in\W_X^{p,*}$ by assumption, all terms in \eqref{oortodox} are in $\W_X^{n-p,*}$ and so
\eqref{oortodox} follows by the SEP if $\tau=0$ on $\Xpreg$. For $\zeta$ in a fixed compact subset of $\Xpreg$,
$B\wedge\debar\chi^\epsilon(z)$ is smooth if $\epsilon$ is small enough. Thus, as above, to see that
$\tau=0$ on $\Xpreg$ it suffices to see that
$\mathcal{R}(\zeta)\debar\chi^\epsilon(z)\wedge\mu(z)\to 0$, which follows if $\debar\chi^\epsilon\wedge\mu\to 0$.
However, since $\debar\mu\in\W_X^{n-p,*}$ by assumption, we have
\begin{equation*}
\debar\chi^\epsilon\wedge\mu = \debar(\chi^\epsilon\mu) - \chi^{\epsilon}\debar\mu \to \debar\mu-\debar\mu=0.
\end{equation*}
\end{proof}

The second lemma we need is (a version of) the crucial Lemma~6.2 in \cite{AS}. 
The proof of that lemma goes through in our case; cf.\ also the proof of \cite[Lemma~5.3]{SK}.
We remark that in these cited lemmas the statements and proofs are intrinsic on (Cartesian products of) $X$ whereas we here formulate 
our version in (Cartesian products of) $D$.  
Let 
\begin{equation*}
k(\zeta,z)=\vartheta\big(B(\zeta,z)\wedge g(\zeta,z)\wedge H(\zeta,z)\big)\mathcal{R}(\zeta).
\end{equation*}
Let $z^j$ be coordinates on the $j$th factor of $D$ in $D\times\cdots\times D$. The current 
\begin{equation}\label{matdax}
T:=\mathcal{R}(z^\nu)\wedge k(z^{\nu-1},z^\nu)\wedge\cdots\wedge k(z^1,z^2)
\end{equation}
is well-defined in $\PM_{D\times\cdots\times D}$, has support in $Z\times\cdots\times Z$ and the SEP
with respect to $Z\times\cdots\times Z$ since it is the product of an almost semi-meromorphic current by 
the tensor product of the $\mathcal{R}$-factors, cf.\ Lemma~\ref{yster}. The different $k$-factors may correspond to different
choices of $B$, $g$, $H$, and $\mathcal{R}$.

\begin{lemma}\label{ASmain}
Let $h$ be a holomorphic tuple which is generically non-vanishing on $Z$. Then 
\begin{equation*}
\lim_{\epsilon\to 0} \debar\chi(|h(z^j)|^2/\epsilon)\wedge T=0.
\end{equation*}
\end{lemma}

\begin{proof}[Proof of Proposition~\ref{hemma}]
It is clear from the definition that $\A_X^{p,*}$ and $\Bsheaf_X^{n-p,*}$ are modules over $\E_X^{0,*}$ and that 
$\A_X^{p,*}$ and $\Bsheaf_X^{n-p,*}$ are closed under $K$-operators and $\check K$-operators, respectively.
By Theorem~\ref{knall} it follows that $\A_X^{p,*}=\E_X^{p,*}$ on $\Xpreg$ and that sections of $\Bsheaf_X^{n-p,*}$ are of the claimed form on $\Xpreg$.

To show that $\A_X^{p,*}\subset\text{Dom}\,\debar_X$ assume that $\varphi$ is given by \eqref{skurk}, where the $\xi_j$ are smooth,
and let $\omega$ be a structure form. Then $i_*\omega=\mathcal{R}$ for some $\mathcal{R}$ and 
$i_*(\omega\wedge\varphi)=\pi^\nu_{*} (T\wedge\xi)$, where $T$ is given by \eqref{matdax}, $\xi$ is some smooth form in $D\times\cdots\times D$,
and $\pi^\nu\colon D\times\cdots\times D\to D$ is the natural projection on the factor with coordinates $z^\nu$. Let $h$ be as in Lemma~\ref{ASmain}
and set $\chi^\epsilon=\chi(|h|^2/\epsilon)$.
By Lemma~\ref{ASmain} we have $\lim_{\epsilon\to 0} \debar\chi^\epsilon\wedge \pi^\nu_{*} T\wedge\xi=0$ and so,
since $i_*(\omega\wedge\varphi)$ has the SEP with respect to $Z$, we get
\begin{eqnarray}\label{smuggla}
\debar i_*(\omega\wedge\varphi) &=& \lim_{\epsilon\to 0}\debar (\chi^\epsilon i_*(\omega\wedge\varphi))
=\lim_{\epsilon\to 0} \debar\chi^\epsilon\wedge \pi^\nu_{*} T\wedge\xi + 
\lim_{\epsilon\to 0} \chi^\epsilon \debar i_*(\omega\wedge\varphi) \nonumber \\
&=&
\lim_{\epsilon\to 0} \chi^\epsilon \debar i_*(\omega\wedge\varphi). 
\end{eqnarray}
Hence, $\debar i_*(\omega\wedge\varphi)$ has the SEP with respect to $Z$ and it follows that $\debar (\omega\wedge\varphi)\in\W_X^{n,*}$.

In a similar way we show that if $\mu\in\Bsheaf_X^{n-p,*}$, then $\debar\mu\in\W_X^{n-p,*}$. Assume that $\mu$ is given by \eqref{tjyv}.
Then $i_*\mu=\pi^\nu_{*}T\wedge\xi$ for some smooth $\xi$. Replacing $\omega\wedge\varphi$ by $\mu$ in \eqref{smuggla} it follows that 
$\debar\mu\in\W_X^{n-p,*}$.

Since $\A_X^{p,*}$ is closed under $K$-operators, $\A_X^{p,*}\subset\text{Dom}\,\debar_X$, and $\A_X^{p,*}=\E_X^{p,*}$ on $\Xpreg$
the Koppelman formula \eqref{ortodox} follows for sections of $\A_X^{p,*}$ by Lemma~\ref{ultraortodox}.
Similarly, since $\Bsheaf_X^{n-p,*}$ is closed under $\check K$-operators, $\debar \Bsheaf_X^{n-p,*}\subset \W_X^{n-p,*}$, and
$\mu=\sum_\ell\mu_\ell\wedge\omega_\ell$ on $\Xpreg$ for any $\mu\in\Bsheaf_X^{n-p,*}$ it follows from 
Lemma~\ref{ultraortodox} that the Koppelman formula \eqref{oortodox} holds for sections of $\Bsheaf_X^{n-p,*}$.

It remains to see that $\A_X^{p,*}$ and $\Bsheaf_X^{n-p,*}$ are closed under $\debar$. 
Let $\varphi\in\A_X^{p,*}$ and assume that $\varphi$ is given by \eqref{skurk}. We show by induction over $\nu$
that $\debar\varphi\in\A_X^{p,*}$. If $\nu=0$, then $\varphi=\xi_0\in\E_X^{p,*}$ and so $\debar\varphi\in\E_X^{p,*}\subset\A_X^{p,*}$.
If $\nu\geq 1$ we write $\varphi=\xi_\nu\wedge K\phi$, where $\phi$ is given by \eqref{skurk} with $\nu$ replaced by $\nu-1$. 
By the induction hypothesis, $\debar\phi\in\A_X^{p,*}$. Hence, $K\debar\phi\in\A_X^{p,*}$. Since the Koppelman formula \eqref{ortodox}, with
$\varphi$ replaced by $\phi$, holds and since $P\phi\in\E_X^{p,*}$ it follows that $\debar K\phi\in\A_X^{p,*}$.
Hence, $\debar\varphi=\debar\xi_\nu\wedge K\phi + \xi_\nu\wedge\debar K\phi\in\A_X^{p,*}$.

If $\mu\in\Bsheaf_X^{n-p,*}$ is given by \eqref{tjyv} we proceed in a similar way by induction over $\nu$. 
If $\nu=0$ then $\mu=\xi_0\wedge\omega$. Then, by the computation showing \eqref{fredag}, it follows that 
$\debar\mu$ has the same form. If $\nu\geq 1$ we write $\mu=\xi_\nu\wedge\check K\mu'$ and the induction hypothesis gives
$\debar\mu'\in\Bsheaf_X^{n-p,*}$. As before, since $\check P\mu'=\xi\wedge\omega$ for some smooth $\xi$, cf.\ Section~\ref{intop}, it follows 
from \eqref{oortodox}, with $\mu$ replaced by $\mu'$,  that $\debar\mu\in\Bsheaf_X^{n-p,*}$.
\end{proof}

\begin{proof}[Proof of Theorem A]
In view of Proposition~\ref{hemma} it only remains to show that \eqref{sol5} is a resolution of $\Om_X^p$.
This is a local statement so we may assume that $X$ is an analytic subspace of a pseudoconvex domain $D\subset\C^N$
and that the point in which we want to show that \eqref{sol5} is exact is $0$.
Let $\varphi\in\A^{p,q}(\mathcal{U}\cap X)$ be $\debar$-closed, where $\mathcal{U}$ is a neighborhood of $0$. 
Choose operators $K$ and $P$ corresponding to a choice of weight $g$ such that $z\mapsto g(\zeta,z)$ is holomorphic in some neighborhood
of $0$ and $\zeta\mapsto g(\zeta,z)$ has support in a fixed compact subset of $\mathcal{U}$.
Then $\vartheta(g\wedge H)\mathcal{R}(\zeta)$ has degree $0$ in $d\bar{z}$. Since it has total bidegree $(N,N)$ it must have full degree
in $d\bar\zeta$. Hence, $P\varphi=0$ if $q\geq 1$. Since, by Proposition~\ref{hemma}, the Koppelman formula \eqref{ortodox} 
holds it follows that $\varphi=\debar K\varphi$
if $q\geq 1$ and $\varphi=P\varphi$ if $q=0$. In the latter case, since $\vartheta(g\wedge H)$ is holomorphic in $z$, we get 
$\varphi\in\Om_X^{p}$.
\end{proof}

\begin{theorem}\label{Bupplosn}
The sheaf complex
\begin{equation}\label{knappar}
0\to \Ba_X^{n-p} \to \Bsheaf_X^{n-p,0} \stackrel{\debar}{\longrightarrow}\cdots\stackrel{\debar}{\longrightarrow} \Bsheaf_X^{n-p,n}\to 0
\end{equation}
is exact if and only if $\Om_X^p$ is Cohen--Macaulay. 
In general, 
\begin{equation*}
\HH^q(\Bsheaf_X^{n-p,\bullet},\debar) \simeq \Ext_{\hol_D}^{\kappa+q}(\Om_X^p,\hol_D).
\end{equation*}
\end{theorem}

\begin{proof}
Consider the free Hermitian resolution \eqref{karvkplx} of $\Om_X^p$ and let $R$ and $\mathcal{R}$ be the associated currents.
Let $(\hol(E^*_\bullet),f^*_\bullet)$ be the dual complex of \eqref{karvkplx}, cf.\ the proof of Proposition~\ref{CHvsR}.
It is well-known that
\begin{equation}\label{knapp}
\HH^{\kappa+ q}(\hol(E^*_\bullet),f^*_\bullet) \simeq \Ext_{\hol_D}^{\kappa+q}(\Om_X^p,\hol_D).
\end{equation}
Define the mapping 
\begin{equation*}
\varrho_\bullet\colon (\hol(E^*_{\kappa+\bullet}),f^*_{\kappa+\bullet}) \to (\Bsheaf_X^{n-p,\bullet},\debar),\quad
i_*\varrho(\xi)=\xi\cdot \mathcal{R}_{\kappa+\bullet}.
\end{equation*}
Since $f\mathcal{R}=\debar\mathcal{R}$ it follows that $\varrho_\bullet$ is a map of complexes and hence induces a 
map on cohomology. As in the proof of \cite[Theorem~1.7]{SK}, cf.\ also the proof of \cite[Theorem~1.2]{RSW},
one shows that the map on cohomology is an isomorphism. Hence, the last statement of the theorem follows. 

If $\Om_X^p$ is Cohen--Macaulay we can choose the free resolution \eqref{karvkplx} 
of length $\kappa$. Thus, by \eqref{knapp}, $\Ext_{\hol_D}^{\kappa+q}(\Om_X^p,\hol_D)=0$ for $q\geq 1$ and so
$\HH^q(\Bsheaf_X^{n-p,\bullet},\debar)=0$ for $q\geq 1$. Hence \eqref{knappar} is exact.
Conversely, if \eqref{knappar} is exact then $\Ext_{\hol_D}^{\kappa+q}(\Om_X^p,\hol_D)=0$ for $q\geq 1$.
Recall the singularity subvarieties $Z_k:=Z_k^{\Om_X^p}$ associated with $\Om_X^p$, cf.\ \eqref{pinne2}.
From, e.g., the proof of \cite[Theorem~II.2.1]{BS} it follows that
\begin{equation*}
Z_{\kappa+q}=\bigcup_{r\geq \kappa+q} \text{supp}\, \Ext_{\hol_D}^r(\Om_X^p,\hol_D).
\end{equation*}
Hence, $Z_{\kappa+q}=\emptyset$ for $q\geq 1$. It follows that $\text{Im}\, f_{\kappa+1}\subset E_\kappa$ is a subbundle.
Replacing $E_\kappa$ by $E_\kappa/\text{Im}\, f_{\kappa+1}$ and $E_{\kappa+q}$, $q\geq 1$, by $0$ in \eqref{karvkplx}
we obtain a free resolution of $\Om_X^p$ of length $\kappa$. Thus, $\Om_X^p$ is Cohen--Macaulay.
\end{proof}

\section{Serre duality}\label{serresektion}
In this section $X$ is a pure $n$-dimensional analytic space. When considering local problems we tacitly assume that $X$ is an
analytic subset of some pseudoconvex domain $D\subset \C^N$.

Let $\varphi\in\A_X^{p,*}$ and $\mu\in\Bsheaf_X^{n-p,*}$. By Proposition~\ref{hemma}, on $\Xpreg$ $\varphi$ is smooth and 
$\mu=\sum_\ell\mu_\ell\wedge\omega_\ell$, where $\mu_\ell$ are smooth and $\omega_\ell\in\Ba_X^{n-p}$. Hence,
$\varphi\wedge\mu$ is well-defined in $\W_X^{n,*}$ on $\Xpreg$.

\begin{theorem}\label{trace}
There is a unique map $\wedge\colon\A_X^{p,*}\times\Bsheaf_X^{n-p,*}\to\W_X^{n,*}$ extending the wedge product on $\Xpreg$. 
If $\varphi\in\A_X^{p,*}$ and $\mu\in\Bsheaf_X^{n-p,*}$, then $\debar(\varphi\wedge\mu)\in\W_X^{n,*}$ and
\begin{equation}\label{leibniz}
\debar (\varphi\wedge\mu)=\debar\varphi\wedge\mu + (-1)^{\text{deg}\,\varphi}\varphi\wedge\debar\mu.
\end{equation}
\end{theorem}

\begin{proof}
This is a local statement.
The uniqueness is clear by the SEP. Moreover, if $\varphi\wedge\mu\in\W_X^{n,*}$ and $\debar(\varphi\wedge\mu)\in\W_X^{n,*}$ for all
$\varphi\in\A_X^{p,*}$ and $\mu\in\Bsheaf_X^{n-p,*}$, then \eqref{leibniz} follows since it holds on $\Xpreg$ and both the left-hand side and
the right-hand side have the SEP.

To show that $\varphi\wedge\mu\in\W_X^{n,*}$ and $\debar(\varphi\wedge\mu)\in\W_X^{n,*}$ we represent
$\varphi$ and $\mu$ by \eqref{skurk} and \eqref{tjyv}, respectively. The case when $\nu=0$ in \eqref{skurk} needs to be handled separately.
In this case $\varphi\in\E_X^{p,*}$ and so
clearly $\varphi\wedge\mu\in\W_X^{n,*}$. Moreover, since by Proposition~\ref{hemma}, $\debar\mu\in\Bsheaf_X^{n-p,*}$
it follows that $\debar(\varphi\wedge\mu)\in\W_X^{n,*}$. 

Assume now that $\nu>0$ in \eqref{skurk}. Then, 
cf.\ \eqref{matdax},
\begin{equation*}
i_*\varphi(\zeta)=\pi_*\big(k(w^\nu,\zeta)\wedge k(w^{\nu-1},w^\nu)\wedge\cdots\wedge k(w^1,w^2)\wedge\xi\big)
=:\pi_* T_\varphi,
\end{equation*} 
where $\pi\colon D_\zeta\times D_{w^\nu}\times\cdots \times D_{w^1}\to D_\zeta$ is the natural projection, $\nu\geq 1$, and $\xi$ is
smooth. Hence, $\varphi=K\phi$ for an appropriate $\phi(w^\nu)\in\A_X^{p,*}$ and so, in view of Section~\ref{intop}, 
since $\mu\in\Bsheaf_X^{n-p,*}\subset\W_X^{n-p,*}$ we have
\begin{equation*}
i_* \mu\wedge \varphi= i_*\mu\wedge K\phi =\pi_*\big(\mu(\zeta)\wedge T_\varphi\big).
\end{equation*}
On $\Xpreg$, where $\varphi$ is smooth, this coincides with the natural wedge product, cf.\ the proof of Theorem~\ref{knall}.
In view of Definition~\ref{Bdef} we may assume that
\begin{equation*}
i_*\mu(\zeta)=\tilde\pi_*\big(\mathcal{R}(z^{\tilde\nu})\wedge k(z^{\tilde\nu -1},z^{\tilde\nu})\wedge\cdots\wedge k(\zeta,z^1)\wedge\tilde\xi\big)
=:\tilde\pi_* T_\mu,
\end{equation*}
where $\tilde\pi\colon D_\zeta\times D_{z^{\tilde\nu}}\times\cdots \times D_{z^1}\to D_\zeta$ is the natural projection and 
$\tilde\xi=\tilde\xi(\zeta,z)$ is smooth.
Hence,
\begin{equation*}
i_*\mu\wedge\varphi = \pi_*\tilde\pi_*\big(T_\mu\wedge T_\varphi\big).
\end{equation*}
Since $T_\mu\wedge T_\varphi$ is of the form \eqref{matdax} (times $\xi\wedge\tilde\xi)$ it follows that $\mu\wedge\varphi\in\W_X^{n,*}$. 
Moreover, by Lemma~\ref{ASmain} and the computation \eqref{smuggla}, with $\omega$ and $T$ replaced by $\mu$ and $T_\mu\wedge T_\varphi$,
respectively,  
it follows that $\debar (\mu\wedge\varphi)$ has the SEP.
\end{proof}

Let $\Bsheaf_c^{n-p,*}(X)$ be the vector space of sections of $\Bsheaf_X^{n-p,*}$ with compact support. By Theorem~\ref{trace} we have,
cf.\ Definition~\ref{cool}, 
a bilinear pairing
\begin{equation}\label{kennel}
\A^{p,q}(X)\times \Bsheaf_c^{n-p,n-q}(X)\to\C\quad
(\varphi,\mu) \mapsto \int_X \varphi\wedge\mu, 
\end{equation}
which only depends on the class of $\mu$ in $H^{n-q}(\Bsheaf_c^{n-p,\bullet}(X),\debar)$ and the class of 
$\varphi$ in $H^q(\A^{p,q}(X),\debar)$. In particular, since $H^0(\A^{p,\bullet}(X),\debar)=\Om^p(X)$
we have a pairing 
$\Om_X^{p}(X)\times H^{n}(\Bsheaf_c^{n-p,\bullet}(X),\debar)\to\C$.

\begin{theorem}\label{lokal}
Assume that $X$ is an analytic subspace of a pseudoconvex domain $D\subset \C^N$. Then
\begin{equation*}
0\to \Bsheaf_c^{n-p,0}(X)\stackrel{\debar}{\longrightarrow}\cdots\stackrel{\debar}{\longrightarrow}\Bsheaf_c^{n-p,n}(X)\to 0
\end{equation*}
is exact except on level $n$. The pairing \eqref{kennel} makes 
$H^{n}(\Bsheaf_c^{n-p,\bullet}(X),\debar)$ the (topological) dual
of $\Om_X^{p}(X)$.
\end{theorem}

Recall that the topology on $\Om^p(X)=\Om^p(D)/\mathcal{J}^p(D)$ is the quotient topology
and that $\Om^p(X)$ is a Fr\'echet space with this topology, see, e.g.,
\cite[Ch.\ IX]{Dem}. Notice that since convergence in $\Om^p(D)$ is uniform convergence on compact subsets, 
a sequence $\varphi_\epsilon\in\Om^p(X)$
converges to $0$ if there are $\tilde\varphi_\epsilon\in\Om^p(D)$ such that $\varphi_\epsilon=i^*\tilde\varphi_\epsilon$
and $\tilde\varphi_\epsilon\to 0$ uniformly on compacts. By the Cauchy estimates it follows that $\tilde\varphi_\epsilon\to 0$
in $\E^{p,0}(D)$.

\begin{proof}
Let $\mu\in \Bsheaf_c^{n-p,n-q}(X)$ be $\debar$-closed. Choose the weight $g$ in the operators $\check K$ and $\check P$ of
Section~\ref{intop} such that $z\mapsto g(\zeta,z)$ is holomorphic in a neighborhood of the holomorphically compact closure of
$\text{supp}\,\mu$ and
$\zeta\mapsto g(\zeta,z)$ has support in a fixed compact for all $z$ in that neighborhood, cf.\ Example~\ref{viktigt}.
Then $\vartheta(g\wedge H)\mathcal{R}(\zeta)$ has degree $0$ in $d\bar{z}$ and so $\check P \mu=0$ if $q\geq 1$, cf.\ \eqref{korka}.
Since by Proposition~\ref{hemma} the Koppelman formula \eqref{oortodox} holds we conclude that $\mu=\debar\check K\mu$
if $q\geq 1$. 
Since $\zeta\mapsto g(\zeta,z)$ has compact support also $\check K\mu$ has and the first statement of the theorem follows.

Suppose now that $\mu\in \Bsheaf_c^{n-p,n}(X)$. Since convergence of a sequence in $\Om^p(X)$ implies convergence
in $\E^{p,0}(D)$ for a suitable sequence of representatives it follows that $\mu$ defines a continuous linear functional $\tilde\mu$
on $\Om^p(X)$ via \eqref{kennel}. This functional only depends on the cohomology class of $\mu$ and so we can,
in view of \eqref{oortodox}, assume that $\mu=\check P\mu$. We have
\begin{equation*}
\check P\mu = \pi^1_{*}\big(\vartheta(g\wedge H)\mathcal{R}(\zeta)\wedge i_*\mu(z)\big) = 
\pm \pi^1_{*}\big(\vartheta(g\wedge H)\wedge i_*\mu(z)\big) \mathcal{R}(\zeta),
\end{equation*}
where $\pi^1\colon D_\zeta\times D_z\to D_\zeta$. Since $g$ and $H$ are holomorphic for $z$ in a neighborhood of 
the holomorphically compact closure of
$\text{supp}\,\mu$ it follows from the Oka-Weil theorem that if $\int_X\varphi\wedge\mu=0$
for all $\varphi\in\Om^p(X)$, then $\pi^1_{*}\big(\vartheta(g\wedge H)\wedge i_*\mu(z)\big)=0$. 
Hence, $\tilde\mu=0$ implies $\mu=0$, i.e., the map $\mu\mapsto \tilde\mu$ is injective. 

To show surjectivity, let $\lambda$ be a continuous linear functional on $\Om^p(X)$. Then $\lambda$ lifts to a 
functional, also denoted $\lambda$, on $\Om^p(D)$. By the Hahn-Banach theorem this functional has to be carried by some 
compact $G\subset D$, which 
we may assume is holomorphically convex,
and there is an $(N-p,N)$-current $\nu$ in $D$ of order $0$ with compact support in a neighborhood $V$ of $G$
such that 
\begin{equation*}
\lambda(\varphi) = \int_D \varphi\wedge\nu,\quad \varphi\in\Om^p(D).
\end{equation*} 
Let $P$ be an operator corresponding to a choice of weight $g$ such that $z\mapsto g(\zeta,z)$ is holomorphic in $V$ and 
$\zeta\mapsto g(\zeta,z)$ has support in a fixed compact subset of $D$ for all $z\in V$. Then, if $\varphi\in\Om^p(X)$,
$P\varphi$ is an extension of $\varphi$ to $V$. Let also $\phi\in\Om^p(D)$ be an arbitrary representative of $\varphi$. Then
\begin{eqnarray*}
\lambda(\varphi) &=& \lambda(P\varphi) = \int_D P\varphi(z)\wedge\nu(z) =
\int_D\pi^2_{*}\big(\vartheta(g\wedge H)\mathcal{R}(\zeta)\wedge\phi(\zeta)\big)\wedge\nu(z) \\
&=&
\big(\vartheta(g\wedge H)\mathcal{R}(\zeta)\wedge\phi(\zeta)\wedge\nu(z)\big) . 1_{D\times D} \\
&=&
\pm \int_D \phi(\zeta)\wedge\pi^1_{*}\big(\vartheta(g\wedge H)\wedge\nu(z)\big)\mathcal{R}(\zeta).
\end{eqnarray*}
Since $\pi^1_{*}\big(\vartheta(g\wedge H)\wedge\nu(z)\big)$ is smooth with compact support in $D$ it follows that there is 
$\mu\in\Bsheaf_c^{n-p,n}(X)$
such that
\begin{equation*}
\pi^1_{*}\big(\vartheta(g\wedge H)\wedge\nu(z)\big)\mathcal{R}(\zeta) = i_*\mu.
\end{equation*}
Hence,
\begin{equation*}
\lambda(\varphi) = \int_D \phi\wedge i_*\mu = \int_X \varphi\wedge\mu.
\end{equation*}
\end{proof}

\begin{theorem}\label{kung}
Let $X$ be a (paracompact) analytic space of pure dimension $n$. If
$H^q(X,\Om_X^p)$ and $H^{q+1}(X,\Om_X^p)$ are Hausdorff, then the pairing
\begin{equation}\label{klocka}
H^q(\A^{p,\bullet}(X),\debar)\times H^{n-q}(\Bsheaf_c^{n-p,\bullet}(X),\debar) \to \C,\quad
([\varphi],[\mu])\mapsto \int_X\varphi\wedge\mu
\end{equation}
is non-degenerate so that $H^{n-q}(\Bsheaf_c^{n-p,\bullet}(X),\debar)$ is the dual of $H^q(\A^{p,\bullet}(X),\debar)$.
\end{theorem}

\begin{proof}[Sketch of proof]
Referring to, e.g., \cite[Section~6.2]{RSW} and \cite[Section~7.3]{SK} for details we outline a proof showing that 
$H^{n-q}(\Bsheaf_c^{n-p,\bullet}(X),\debar)$ is the dual of $H^q(\A^{p,\bullet}(X),\debar)$
via \eqref{klocka}. The idea is to use \v{C}ech cohomology and homological algebra to reduce to the local duality of Theorem~\ref{lokal}.

Let $\mathcal{U}=\{U_j\}_j$ be a locally finite covering of $X$ such that $U_j$ is an analytic subspace of some pseudoconvex domain $D$ 
in some $\C^N$. Then $\mathcal{U}$ is a Leray covering for $\Om^p_X$. Let $(C^\bullet(\mathcal{U},\Om^p_X),\delta)$
be the associated \v{C}ech cochain complex. Then
\begin{equation}\label{iso1}
H^q(\A^{p,\bullet}(X),\debar) \simeq H^q(C^\bullet(\mathcal{U},\Om^p_X),\delta)
\end{equation}
since both the left- and the right-hand sides are isomorphic to $H^q(X,\Om_X^p)$. It is standard to make the 
isomorphism \eqref{iso1} explicit by solving local $\debar$-equations. 

The Fr\'echet topology on $\Om^p(U_j)$ induces a natural Fr\'echet topology on $C^k(\mathcal{U},\Om^p_X)$
and, consequently, on the cohomology of $(C^\bullet(\mathcal{U},\Om^p_X),\delta)$.
Recall that the standard topology on $H^q(X,\Om_X^p)$ is defined as this topology.
In view of, e.g., \cite[Lemma~2]{RaRu} it follows that if 
$H^q(X,\Om_X^p)$ and $H^{q+1}(X,\Om_X^p)$ are Hausdorff, then 
\begin{equation}\label{iso2}
H^q(C^\bullet(\mathcal{U},\Om^p_X),\delta)^*\simeq H^q(C^\bullet(\mathcal{U},\Om^p_X)^*,\delta^*),
\end{equation}
where $(C^\bullet(\mathcal{U},\Om^p_X)^*,\delta^*)$ is the (topological) dual complex of 
$(C^\bullet(\mathcal{U},\Om^p_X),\delta)$.

Let $C^{-k}(\mathcal{U},\Bsheaf_c^{n-p,j})$ be the group of formal sums $\sum'_{|I|=k+1}\mu_I U^*_I$, where 
$\mu_I\in \Bsheaf_c^{n-p,j}(\cap_{i\in I}U_i)$ and $U^*_I:=U^*_{i_0}\wedge\cdots\wedge U^*_{i_k}$ is a formal
wedge product. It follows from Theorem~\ref{lokal} that
\begin{equation}\label{iso3}
C^k(\mathcal{U},\Om^p_X)^*\simeq H^n(C^{-k}(\mathcal{U},\Bsheaf_c^{n-p,\bullet}),\debar)
\end{equation}
via the pairing induced by \eqref{kennel}. The operator $\delta^*$ on $C^\bullet(\mathcal{U},\Om^p_X)^*$
gives in a natural way an operator, also denoted $\delta^*$, on $C^{-\bullet}(\mathcal{U},\Bsheaf_c^{n-p,j})$.
It turns out that this operator is formal interior multiplication by $\sum_\ell U_\ell$; $\mu_I$ is extended
to $\cap_{i\in I\setminus i_\ell}U_i$ by $0$.
Thus we get the double complex
\begin{equation}\label{double}
\big(C^{-\bullet}(\mathcal{U},\Bsheaf_c^{n-p,\bullet}), \delta^*,\debar\big).
\end{equation}
In view of \eqref{iso3} we have
\begin{equation}\label{iso4}
H^q(C^\bullet(\mathcal{U},\Om^p_X)^*,\delta^*) \simeq
H^q(H^n(C^{-\bullet}(\mathcal{U},\Bsheaf_c^{n-p,\bullet}),\delta^*,\debar)).
\end{equation}
By Theorem~\ref{Bupplosn}, the $\debar$-cohomology of \eqref{double} is trivial except on level $n$ and, by, e.g., \cite[Lemma~6.3]{RSW},
since the $\Bsheaf_X$-sheaves are fine, the $\delta^*$-cohomology of \eqref{double} is trivial except on level $0$ where the cohomology
is $\Bsheaf_c(X)^{n-p,\bullet}$. By standard homological algebra it follows that 
\begin{equation}\label{iso5}
H^q(H^n(C^{-\bullet}(\mathcal{U},\Bsheaf_c^{n-p,\bullet}),\delta^*),\debar) \simeq
H^{n-q}(\Bsheaf_c(X)^{n-p},\debar).
\end{equation}
From \eqref{iso1}, \eqref{iso2}, \eqref{iso4}, and \eqref{iso5} we see that $H^{n-q}(\Bsheaf_c(X)^{n-p},\debar)$ is the dual
of $H^q(\A^{p,\bullet}(X),\debar)$. To see that this duality is given by \eqref{klocka} one can make these isomorphisms explicit 
and use that \eqref{iso3} is induced by the pairing
\eqref{kennel}. 
\end{proof}

\begin{proof}[Proof of Theorem B]
Part (i) follows from Definition~\ref{Bdef} and Proposition~\ref{hemma}. 
Part (ii) follows from Theorem~\ref{Bupplosn}. 
Part (iii) follows from Theorem~\ref{trace}.
Part (iv) follows from Theorem~\ref{kung}; indeed, if $X$ is compact then we can replace $\Bsheaf_c^{n-p,\bullet}(X)$ by $\Bsheaf^{n-p,\bullet}(X)$
and, moreover, by the Cartan-Serre theorem, the cohomology of any coherent sheaf is finite-dimensional, in particular Hausdorff.
\end{proof}

\section{Examples}\label{Exsektion}
We present two examples which illustrate our various notions of holomorphic forms and currents.

\begin{example}\label{matlen}
Let $D=\C^{4}$ with coordinates $(z_{1},z_{2},w_{1},w_{2})$ and let $\mathcal{J}=\big\langle w_{1}^{2},w_{2}^{2},w_{1}w_{2}\big\rangle$. 
Then $d\mathcal{J}=\big\langle w_{1}d w_{1},w_{2}d w_{2}, d(w_{1}w_{2})\big\rangle$ and 
$\sqrt{\J}=\langle w_{1},w_{2}\rangle$ so that $Z=\{w=0\}$.  
 Clearly, $\hol_{X} = \hol_{Z}\{1,w_{1},w_{2}\}$. By Taylor expansion of the coefficients of elements of $\Om_D^1$
it follows that $\Om^{1}_{X,\text{K\"{a}hler}}=\Om^{1}_{Z}\{1,w_{1},w_{2}\}+\hol_{Z}\{d w_{1},d w_{2},w_{1}d w_{2},w_{2}dw_{1}\}$.
Since $w_1dw_2=-w_2dw_1+d(w_1w_2)$ the number of generators can be reduced and since $w_1dw_2-w_2dw_1=2w_1dw_2- d(w_1w_2)$  we get
\begin{equation*}
\Om^{1}_{X,\text{K\"{a}hler}}=\Om^{1}_{Z}\{1,w_{1},w_{2}\}+\hol_{Z}\{d w_{1},d w_{2},w_{1}d w_{2}-w_{2}dw_{1}\}.
\end{equation*}
Since $\J$ is independent of $z$ it follows that $\Om^{1}_{X,\text{K\"{a}hler}}$ is torsion-free, and so $\Om_X^1=\Om^{1}_{X,\text{K\"{a}hler}}$.
In a similar way,
\begin{equation*}
\Om^{2}_{X} 
=
\Om^{2}_{Z}\{1,w_{1},w_{2}\}+\Om^{1}_{Z}\wedge
\{d w_{1},d w_{2},w_{1}d w_{2}-w_{2}dw_{1}\} +\hol_{Z}\{d w_{1}\wedge d w_{2}\}.
\end{equation*}

In view of Proposition~\ref{rutschkana} there is an analogous description of the smooth forms on $X$. For instance,
\begin{equation*}
\mathscr{E}^{2,*}_{X}=\mathscr{E}^{2,*}_{Z}\{1,w_{1},w_{2}\}+\mathscr{E}^{1,*}_{Z}\wedge\{d w_{1},d w_{2},w_{1}d w_{2}-w_{2}d w_{1}\}
+\mathscr{E}^{0,*}_{Z}\wedge\{d w_{1}\wedge d w_{2}\}.
\end{equation*}

Now let us look at currents on $X$; for simplicity we restrict to currents of bidegree $(2,*)$.
If $\alpha$ is a $(2,*)$-current on $X$ then, by definition, $i_{*}\alpha$ is annihilated by $\Kernel_0i^*$, which contains all
$\bar{w}_{i}$ and $d\bar{w}_{i}$. It follows that
\begin{equation*}
i_{*}\alpha=\sum_{k,\ell\geq 0}\alpha_{k,\ell}(z) d z_{1}\wedge d z_{2}\wedge d w_{1}\wedge d w_{2}\wedge\debar\frac{1}{w_{1}^{k+1}}\wedge
\debar\frac{1}{w_{2}^{\ell+1}}\quad\text{with }\alpha_{k,\ell}\in\mathscr{C}^{0,*}_{Z},
\end{equation*}
where $\mathscr{C}_Z$ is the sheaf of currents on $Z$; cf.\ \eqref{grymta}.
However, we must also have $i_{*}\alpha\wedge\mathcal{J}=0$ and $i_{*}\alpha\wedge d\mathcal{J}=0$. The first equality 
implies that $k,\ell\leq 1$ and the second is automatically satisfied for degree reasons. 
Moreover, we have $w_{1}w_{2}\debar(d w_{1}/w_{1}^{2})\wedge\debar(d w_{2}/w_{2}^{2})\neq0$ and therefore
\begin{equation*}
i_{*}\mathscr{C}^{2,*}_{X}=\mathscr{C}^{2,*}_{Z}\bigg\{\debar\frac{d w_{1}}{w_{1}}\wedge\debar\frac{d w_{2}}{w_{2}},\,
\debar\frac{d w_{1}}{w_{1}^{2}}\wedge\debar\frac{d w_{2}}{w_{2}},\,\debar\frac{d w_{1}}{w_{1}}\wedge\debar\frac{d w_{2}}{w_{2}^{2}}\bigg\}.
\end{equation*}
Writing $\mathcal{B}$ for the set of three basis elements above, we get $i_{*}\Ba^{2}_{X}=\Om^{2}_{Z}\mathcal{B}$.
\end{example}

\begin{example}
Let $D=\C^{4}$ with the same coordinates as above. Let $i\colon X\to D$ be defined by 
\begin{equation*}
\mathcal{J}=\big\langle w_{1}^{2},w_{2}^{2},w_{1}w_{2},z_{1}w_{2}-z_{2}w_{1}\big\rangle,
\end{equation*} 
and write $f=z_{1}w_{2}-z_{2}w_{1}$ so that $d\mathcal{J}=\big\langle w_{1}d w_{1},w_{2}d w_{2},d(w_1w_2),d f\big\rangle$. 
Notice that except for the extra generator $f$, $\J$ is the same as in Example~\ref{matlen}. Thus,
\begin{equation*}
\hol_{X}=\frac{\hol_{Z}\{1,w_{1},w_{2}\}}{\hol_{Z}\{f\}},\quad
\Om^{1}_{X,\text{K\"{a}hler}}=
\frac{\Om^{1}_{Z}\{1,w_{1},w_{2}\}+\hol_{Z}\{d w_{1},d w_{2},w_{1}d w_{2}-w_{2}d w_{1}\}}{\Om^{1}_{Z}\{f\}+\hol_{Z}\{d f\}}.
\end{equation*}
Here we write the quotient as a quotient of $\hol_{Z}$-modules.
The $\hol_X$-module structure is obtained by imposing $w_{i}w_{j}=0$
and $w_{i}dw_{i}=0$.

We now describe the torsion elements of $\Om^{1}_{X,\text{K\"{a}hler}}$. If $z_{1}\neq 0$ then $w_{2}=w_{1}z_{2}/z_{1}$ and 
since $w_{1}d w_{1}=0$ we get $w_{2}d w_{1}=\big(z_{2}w_{1}/z_{1}\big)d w_{1}=0$. We have $w_{1}d w_{2}+w_{2}d w_{1}=0$ 
everywhere and therefore we also get $w_{1}d w_{2}=0$ when $z_{1}\neq0$. By symmetry both $w_{1}d w_{2}$ and $w_{2}d w_{1}$ 
vanish when $z_{2}\neq0$. One may verify that neither $w_{1}d w_{2}$ nor $w_{2}d w_{1}$ is in $\Omega^{1}_{Z}\{f\}+\hol_{Z}\{d f\}$ 
and therefore they are torsion elements. 
Moreover, one can check that $\Om^{1}_{X,\text{K\"{a}hler}}/\hol_Z\{w_1dw_2, w_2dw_1\}$ is torsion-free and hence
\begin{equation*}
\Om^{1}_{X}=\frac{\Om^{1}_{Z}\{1,w_{1},w_{2}\}+\hol_{Z}\{d w_{1},d w_{2}\}}{\Om^{1}_{Z}\{f\}+\hol_{Z}\{d f\}}.
\end{equation*}
Similar calculations yield that
\begin{equation*}
\Om^{2}_{X}=\frac{\Om^{2}_{Z}\{1,w_{1},w_{2}\}+\Om^{1}_{Z}\{d w_{1},d w_{2}\}}{\Om^{2}_{Z}\{f\}+\Om^{1}_{Z}\{d f\}}.
\end{equation*}

We now describe generators for $\Ba^{2-p}_X$. In \cite[Example~6.9]{AL},
it was shown that the generators for $i_* \Ba^2_X$ are given by
\begin{equation*}
    \debar\frac{1}{w_1}\wedge \debar\frac{1}{w_2} \wedge dz \wedge dw \text { and }
     \left(z_1 \debar\frac{1}{w_1^2}\wedge \debar\frac{1}{w_2}+z_2\debar\frac{1}{w_1}\wedge \debar\frac{1}{w_2^2}\right) \wedge dz \wedge dw,
\end{equation*}
where $dz = dz_1\wedge dz_2$ and $dw = dw_1 \wedge dw_2$.
These correspond to intrinsic objects in $\Ba^2_X$, which can be considered as differential operators, and in view of the formula
\begin{equation*}
    \frac{1}{(2\pi i)^2} \debar\frac{1}{w_1^{m_1+1}} \wedge \debar\frac{1}{w_2^{m_2+1}} . \psi(w) = \left.\frac{1}{m_1! m_2!} \frac{\partial^{m_1+m_2}}{\partial w_1^{m_1}\partial w_2^{m_2}} \psi \right|_{w=0},
\end{equation*}
they correspond (up to constants) to the form-valued differential operators
\begin{equation*}
    dz_1 \wedge dz_2 \text{ and } dz_1 \wedge dz_2 \left(z_1\frac{\partial}{\partial w_1}+z_2 \frac{\partial}{\partial w_2}\right)
\end{equation*}
followed by restriction to $Z$.

By a similar calculation as in \cite[Example~6.9]{AL}, one can obtain also generators for $i_* \Ba^{2-p}_X$ for $p=1,2$.
Indeed, if $(E,f)$ is a locally free Hermitian resolution of $\Om^p_X$, then $\Ba^{2-p}_X$ is generated by all currents of the
form $\xi R^E_2$, where $\xi$ is in $\Kernel f_3^*$ and $R^E_2$ is the part in $\text{Hom}(E_0,E_2)$ of the residue current associated to $(E,f)$.
To calculate $\xi R^E_2$, by the same argument as in \cite[Example~6.9]{AL}, if $(F,g)$ is the direct sum of $r_0$  
copies of the Koszul complex of $(w_1^2,w_2^2)$, where $r_0 = \text{rank}\, E_0 = \text{rank}\, \Om^p_{\C^4}$, 
and $a\colon (F,g) \to (E,f)$ is a morphism of complexes such that
$a_0\colon F_0 \to E_0$ is the identity, then
$$\xi R^E_2 h = \xi a_2(he) \wedge \debar \frac{1}{w_2^2} \wedge \debar \frac{1}{w_1^2}.$$ 
With the help of the software Macaulay2, we could calculate the morphism $a_2$ and generators for $\Kernel f_3^*$, 
and could thus calculate generators for $i_* \Ba^{2-p}_X$.
The sheaf $i_* \Ba^1_X$ is generated by
\begin{eqnarray*}
& & dz_1 \wedge dw_1 \wedge dw_2 \wedge \debar\frac{1}{w_2} \wedge \debar\frac{1}{w_1}, \quad 
dz_2 \wedge dw_1 \wedge dw_2 \wedge \debar\frac{1}{w_2} \wedge \debar\frac{1}{w_1}, \\
& &
\left((z_2 w_1+z_1 w_2) dz_2 \wedge dw_1\wedge dw_2 + w_1 w_2 dz_1 \wedge dz_2 \wedge dw_2 \right)\wedge \debar \frac{1}{w_2^2} 
\wedge \debar \frac{1}{w_1^2}, \\ 
& &
\left((z_2 w_1+z_1 w_2) dz_1 \wedge dw_1\wedge dw_2 + w_1 w_2 dz_1 \wedge dz_2 \wedge dw_1 \right)\wedge \debar \frac{1}{w_2^2} 
\wedge \debar \frac{1}{w_1^2}, \\
& &
\left(z_2 dz_1 \wedge dz_2 \wedge dw_1 - z_1 dz_1 \wedge dz_2 \wedge dw_2\right) \wedge \debar\frac{1}{w_2} \wedge \debar\frac{1}{w_1}.
\end{eqnarray*}
These correspond (up to constants) to the differential operators
\begin{eqnarray*}
& &
dz_1, \quad dz_2, \\ 
& &
dz_2\left( z_2 \frac{\partial}{\partial w_2} + z_1 \frac{\partial}{\partial w_1}\right) +
    dz_1 \wedge dz_2 \wedge (dw_1 \lrcorner),  \\
& &
dz_1\left( z_2 \frac{\partial}{\partial w_2} + z_1 \frac{\partial}{\partial w_1}\right) +
    dz_1 \wedge dz_2 \wedge (dw_2 \lrcorner), \\
& &
z_2 dz_1 \wedge dz_2 \wedge (dw_2 \lrcorner) - z_1 dz_1 \wedge dz_2 \wedge (dw_1 \lrcorner)
\end{eqnarray*}
followed by restriction to $Z$.
Finally, the sheaf $i_* \Ba^0_X$ is generated by 
\begin{eqnarray*}
& &    
dw_1 \wedge dw_2 \wedge \debar\frac{1}{w_2} \wedge \debar\frac{1}{w_1}, \\ 
& &
\left((z_2 w_1 + z_1 w_2) dw_1 \wedge dw_2 + w_1 w_2 dz_2 \wedge dw_1 - w_1 w_2 dz_1 \wedge dw_2\right) \wedge 
\debar\frac{1}{w_2^2} \wedge \debar\frac{1}{w_1^2}, \\
& &
\left( z_1 dz_1 \wedge dw_2 - z_2 dz_1 \wedge dw_1\right) \wedge \debar\frac{1}{w_2} \wedge \debar\frac{1}{w_1}, \\ 
& &
\left( z_1 dz_2 \wedge dw_2 - z_2 dz_2 \wedge dw_1\right) \wedge \debar\frac{1}{w_2} \wedge \debar\frac{1}{w_1}.
\end{eqnarray*}
These correspond (up to constants) to the differential operators
\begin{eqnarray*}
& &
1, \quad z_2 \frac{\partial}{\partial w_2} + z_1 \frac{\partial}{\partial w_1} + dz_2 \wedge (dw_2 \lrcorner) + dz_1 \wedge (dw_1 \lrcorner), \\
& &
dz_1 \wedge (z_1 dw_1\lrcorner + z_2 dw_2\lrcorner), 
\quad dz_2 \wedge (z_1 dw_1\lrcorner + z_2 dw_2\lrcorner)
\end{eqnarray*}
followed by restriction to $Z$.
\end{example}

\begin{remark}
Finally let us put the calculations of $\Ba_X^\bullet$ into the context of Noetherian differential operators. 
Let as before $i\colon X\to D\subset\C^N$ be defined by $\J$ and $Z=Z(\J)$.
Recall that a holomorphic differential operator $\mathcal{L}\colon \hol_D\to\hol_Z$ is 
Noetherian for $\J$ if $\mathcal{L}\varphi=0$ for any $\varphi\in\J$.
A set $\{\mathcal{L}_j\}_j$ of such operators is complete if $\varphi\in\J$ if and only if 
$\mathcal{L}_j\varphi=0$ for all $j$. 

Assume that $Z$ is smooth and that $(z,w)$ are coordinates such that $Z=\{w=0\}$ and let $\pi\colon D\to Z$ be the projection
$\pi(z,w)=z$. Given a set of generators $\mu=(\mu_1,\ldots,\mu_m)$ of $\Ba_X^{n-p}$ we construct a complete set of Noetherian type operators
$\Om_D^p\to\Om_Z^n$ (acting as Lie derivatives)
for $\J^p$ in a way similar to the construction of the mapping $\widetilde T$ in Section~\ref{smooth-reg}. Take $M>0$ large enough so that
$w^\alpha\mu_j=0$ for all $j$ if $|\alpha|\geq M$. We set
\begin{equation*}
\mathcal{L}_{j,\alpha}\colon \Om_D^p\to \Om_Z^n,\quad \mathcal{L}_{j,\alpha}\varphi=\pi_* (w^\alpha\varphi\wedge i_*\mu_j);
\end{equation*}  
that $\mathcal{L}_{j,\alpha}\varphi\in \Om_Z^n$ follows as in Section~\ref{smooth-reg}. Moreover, $\varphi\wedge i_*\mu_j=0$ if and only if
$\pi_* (w^\alpha\varphi\wedge i_*\mu_j)=0$ for all $\alpha$. Since $\Kernel_pi^*$ is the annihilator of $\mu$ and 
$\Kernel_p i^*\cap \Om_X^p = \J^p$, see the proof of Corollary~\ref{gurka},
it follows that $\mathcal{L}_{j,\alpha}$, $j=1,\ldots,m$, $|\alpha|<M$, is a complete set of Noetherian operators for $\J^p$.
\end{remark}

\end{document}